\documentclass[final,leqno,letterpaper]{etna}

\usepackage{microtype}
\usepackage{xcolor}
\usepackage{amsmath,amssymb}
\usepackage{mathtools}
\usepackage{braket}
\usepackage{jlcode}
\usepackage{algorithmic}
\usepackage{jabbrv}

\usepackage{graphicx}
\makeatletter
\newsavebox\pandoc@box
\newcommand*\pandocbounded[1]{%
  \sbox\pandoc@box{#1}%
  \Gscale@div\@tempa{\textheight}{\dimexpr\ht\pandoc@box+\dp\pandoc@box\relax}%
  \Gscale@div\@tempb{\linewidth}{\wd\pandoc@box}%
  \ifdim\@tempb\p@<\@tempa\p@\let\@tempa\@tempb\fi%
  \ifdim\@tempa\p@<\p@\scalebox{\@tempa}{\usebox\pandoc@box}%
  \else\usebox{\pandoc@box}%
  \fi%
}
\def\fps@figure{htbp}
\makeatother

\usepackage{stmaryrd} %

\title{Multiscale methods for discretized continuous optimization: convergence and cost analysis\thanks{This work was funded by the Natural Sciences and Engineering Research Council of Canada. Manuscript Date: July 15, 2026.}}
\author{Nicholas J. E. Richardson\thanks{Department of Mathematics, University of British Columbia, Vancouver, BC, Canada} \and Noah Marusenko\thanks{Department of Computer Science, University of British Columbia, Vancouver, BC, Canada} \and \mbox{Michael P. Friedlander\footnotemark[2] \footnotemark[3]}}

\newcommand{\half}{\tfrac{1}{2}}
\newcommand{\coarse}[1]{\overline{#1}}
\newcommand{\interp}[1]{\underline{#1}}
\newcommand{\lipschitz}[1]{L_{#1}}
\newcommand{\smoothness}[1]{\mathcal{S}_{#1}}
\newcommand{\convexity}[1]{\mu_{#1}}
\newcommand{\norm}[1]{\left\lVert #1 \right\rVert}
\newtheorem{remark}[theorem]{Remark}
\newcommand\scalemath[2]{\scalebox{#1}{\mbox{\ensuremath{\displaystyle #2}}}}

\usepackage{cleveref}
\crefformat{equation}{(#2#1#3)}
\Crefformat{equation}{Equation~(#2#1#3)}
\crefname{theorem}{Theorem}{Theorems}
\crefname{definition}{Definition}{Definitions}
\crefname{lemma}{Lemma}{Lemmas}
\crefname{corollary}{Corollary}{Corollaries}
\crefname{proposition}{Proposition}{Propositions}
\crefname{remark}{Remark}{Remarks}
\crefname{section}{Section}{Sections}
\Crefname{section}{Section}{Sections}
\crefname{subsection}{Section}{Sections}
\Crefname{subsection}{Section}{Sections}
\crefname{figure}{Figure}{Figures}
\Crefname{figure}{Figure}{Figures}
\crefname{table}{Table}{Tables}
\Crefname{table}{Table}{Tables}
\crefname{algorithm}{Algorithm}{Algorithms}
\Crefname{algorithm}{Algorithm}{Algorithms}
\usepackage{algorithm}

\setbibdata{1}{xx}{46}{2026} %

\hypersetup{
  pdftitle={Multiscale Methods for Discretized Continuous Optimization: Convergence and Cost Analysis},
  pdfauthor={Nicholas~J. E. Richardson, Noah Marusenko, Michael P. Friedlander},
pdfkeywords={multiresolution analysis, continuous optimization, multigrid methods, approximation theory}
}

\shorttitle{MULTISCALE METHODS FOR DISCRETIZED OPTIMIZATION} 
\shortauthor{N.~J.~E.~RICHARDSON, N.~MARUSENKO, AND M.~P.~FRIEDLANDER}

\begin{document}

\maketitle

\begin{abstract}
Discretized versions of optimization problems over continuous arguments are routinely solved
at a single fine resolution, incurring a per-iteration cost that grows, often
superlinearly, with the number of grid points. This paper analyzes a
multiscale method that instead solves a hierarchy of increasingly fine dyadic
discretizations. Linear interpolation of
each coarse solution warm starts the next finer scale using any $q$-linearly
convergent update rule as the inner solver. Each coarse problem is a
consistent discretization of the continuous problem. Structural properties such
as convexity and smoothness are preserved. For problems with Lipschitz-continuous
solutions, two variants of the method converge to the fine-scale solution with explicit error bounds.
The fine-scale solution in turn approximates the continuous solution once the
grid is sufficiently fine, with quantified constants. The total cost to reach a
fixed accuracy is provably lower than that of single-scale optimization
whenever the cost of one update grows at least linearly in the problem size.
Numerical experiments on probability density demixing problems, including
geological survey data, show four- to sevenfold speedups while using a fraction
of the memory.
\end{abstract}

\begin{keywords}
multiresolution analysis, continuous optimization, multigrid methods, approximation theory
\end{keywords}

\begin{AMS}
65B99, 65D15, 65K10, 90C59
\end{AMS}

\section{Introduction}\label{introduction}

Many finite-dimensional optimization problems are discretizations of continuous ones. In density estimation, signal reconstruction, and inverse problems, the vector of unknowns $x\in\mathbb{R}^I$ represents samples of an implicit continuous function $f:\mathcal{D}\to\mathbb{R}$. The finite-dimensional problem
\begin{equation}\phantomsection\label{eq-discretized-problem}{
  \min_{x}\set{ \tilde{\mathcal{L}} (x) | x \in \tilde{\mathcal{C}}}
}
\end{equation}
stands in for an underlying continuous problem
\begin{equation}\phantomsection\label{eq-abstract-continuous-problem}
  \min_{f}\set{\mathcal{L} (f) | f\in\mathcal{C}}, \tag{P}
\end{equation}
posed over a class $\mathcal{C}$ of continuous functions on a one-dimensional domain $\mathcal{D}\subseteq\mathbb{R}$. The discretized objective $\tilde{\mathcal{L}}$ and constraint set $\tilde{\mathcal C}$ correspond to their continuous counterparts, with $x[i]=f(t[i])$ on a grid $\{t[i]\}_{i\in I}\subset\mathcal{D}$. Throughout, \emph{discretization} means a finite-dimensional reduction of the problem—including the variable, objective, and constraints—obtained by sampling $f$ on a grid, rather than the approximation of a differential operator.

Fine grids favor accuracy, but carry computational costs: most operations slow and consume more memory as the grid size $I$ grows, often superlinearly, such as with matrix-matrix products~\cite{strassen_gaussian_1969}. Standard practice incurs this cost at every iteration by solving~\eqref{eq-discretized-problem} at the finest scale, from the first iterate to the last. That approach, however, discards the continuity of the underlying solution. Samples of a continuous function at neighbouring grid points are themselves close, so a solution computed on a course grid carries useful information needed at a finer grid.

This paper analyzes a multiscale method that exploits the continuity underlying the discretized problem. The method solves a hierarchy of dyadic discretizations of~\cref{eq-abstract-continuous-problem}, coarsest first, linearly interpolating each coarse solution to warm start the next finer scale. The method runs any iterative update rule with at least $q$-linear iterate convergence as the inner solve at each scale. We prove that this schedule reaches the fine-scale solution, and thus a bounded approximation to the continuous problem, at provably lower total cost than solving at the finest scale alone. \Cref{fig-multi-vs-single-time} previews the effect on a density recovery problem: past a modest problem size, the multiscale schedule dominates, and its efficiency improves as the problem size grows.

\begin{figure}[t]
\centering
\includegraphics[width=0.55\linewidth]
{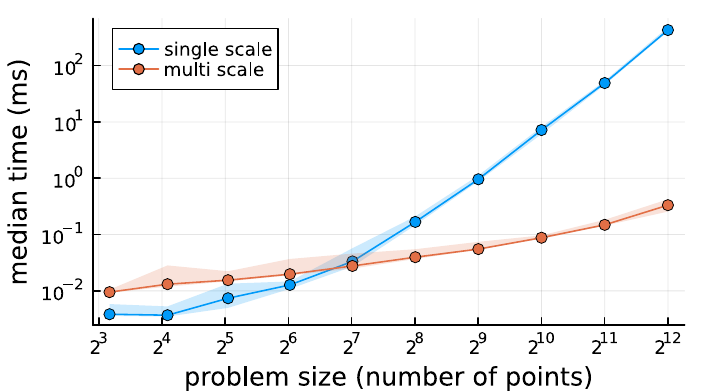}
\caption{Comparison of multiscale vs single-scale approach for the motivating example from \cref{sec-motivating-example}. Dots represent the median total time in milliseconds over \(100\) trials; shaded regions represent the \(5\)th and \(95\)th percentile times. The multiscale approach improves algorithm speed over a single scaled approach once the problem size is large enough. Details are provided in \cref{sec-motivating-example-numerics}.}\label{fig-multi-vs-single-time}
\end{figure}

\subsection{Contributions}

We prove this advantage holds under a simple discretization scheme, and any inner algorithm satisfying mild conditions, adding rigour to the \emph{ad hoc} approach often used in practice.

\begin{itemize}

\item \emph{Consistent construction (\cref{the-multiscale-optimization-method}}). Dyadic coarsening, midpoint linear interpolation, and explicit rescaling rules for linear and norm constraints ensure that at every scale $s$, the finite dimensional approximation \cref{eq-scale-s-discretized-problem} discretizes the same continuous problem \cref{eq-abstract-continuous-problem}. Each discretization preserves convexity of constraints and smoothness or strong convexity of objectives, each with correctly scaled constants.

\item \emph{Convergence to fine-scale solution} (\cref{sec-multiscale-convergence}). Two variants are analyzed: a \emph{greedy} variant, which re-optimizes all variables at each scale, and a \emph{lazy} variant, which freezes coarse values and optimizes only newly interpolated points. For both, the final error decomposes as a reduced initial error plus an accumulated interpolation error; Theorems~\ref{thm-greedy-multiscale-descent-error} and \ref{thm-lazy-multiscale-descent-error} make each term explicit and apply to any inner solver with $q$-linear iterate convergence (\cref{def-iterate-convergence}), including projected gradient descent.

\item \emph{Fine-scale approximation of the continuous problem} (\cref{sec-relation-between-the-discretized-and-continuous-problems}). A small error in the discretized solution implies a small error between the continuous solution and its piecewise linear reconstruction, once the grid size exceeds an explicit threshold (\cref{thm-problem-connection}).

\item \emph{Multiscale schedule efficiency} (\cref{comparison-with-projected-gradient-descent}). When one update at problem size $I$ costs $\Theta(I^p)$ with $p\ge1$, explicit iteration schedules make greedy and lazy multiscale simultaneously cheaper and tighter in expected error than single-scale projected gradient-descent, provided the problem is large enough (\cref{cor-greedy-cost,cor-lazy-cost}).

\item \emph{Predicted speedup in practice} (\cref{sec-numerical-experiments}). On synthetic and real geological density demixing problems, which are tensor-valued problems beyond the 1-dimensional theory, the greedy multiscale method runs four to seven times faster than its single-scale counterpart, with a fraction of the memory.

\end{itemize}

\begin{remark}[Scope]\label[remark]{rem-scope}
We state results for scalar functions on $\mathcal{D}=[0,1]$. This is without loss of generality: an affine change of variables maps any compact interval $[\ell, u]$ to $[0,1]$ and replaces the Lipschitz constant $L_f$ by the effective constant $(u-\ell)L_f$, which is the quantity that contributes to every bound. Tensor-valued functions are handled by vectorization, as described in \cref{sec-numerical-experiments}. The restriction to Lipschitz-continuous solutions is deliberate: the convergence analysis needs only a bound on how much $f$ can vary between grid points (a modulus of continuity) and analogues of our results hold for uniformly continuous functions. The Lipschitz assumption yields explicit constants that can be weighed against computational cost, and the applications we target have Lipschitz solutions. The problem~\eqref{eq-abstract-continuous-problem} could be posed over a Banach space of functions, but the questions we address regarding iteration complexity, cost, and grid sizes are concrete and so we choose a matching concrete setting.
\end{remark}

\subsection{Related work}\label{related-work}

Our method is an instance of \emph{nested iteration}, the coarse-to-fine
strategy underlying the full multigrid method~\cite{trottenberg_multigrid_2001}
and the cascadic multigrid method of Bornemann and
Deuflhard~\cite{bornemann1996cascadic}: solve on a coarse grid,
interpolate, warm start the next finer grid, and never return to coarser
grids. Cascadic multigrid applies this template to linear elliptic
equations with a fixed linear smoother. Here, we apply it to constrained
optimization with a generic $q$-linearly convergent update, and we ask a
different question: for a fixed accuracy at the finest scale, when is the
total cost provably lower than solving at that scale alone?

A second line of work adapts the recursive V-cycle structure of multigrid
to optimization: the MG/OPT framework of Nash~\cite{nash_multigrid_2000},
recursive multilevel trust-region methods of Gratton, Sartenaer, and
Toint~\cite{gratton_recursive_2008}, multigrid methods for PDE-constrained
optimization surveyed by Borz{\`i} and Schulz~\cite{borzi_multigrid_2009},
and the proximal variant MGProx~\cite{ang_mgprox_2024}. These methods start
at the finest scale and recursively invoke coarse corrections to accelerate
fine-scale steps. Our method differs structurally: it passes through the
scales once, coarse to fine, and treats the inner solver as a black
box---and differs in what is proved: rather than per-iteration descent of a
fine-scale merit function, we bound the \emph{total cost to a fixed
accuracy} and give sufficient conditions under which the multiscale
schedule dominates the single-scale one. The algorithmic template is standard. Our
contribution is the analysis, which rests on tight interpolation-error
bounds for Lipschitz functions
(\cref{lem-lipschitz-interpolation,lem-exact-interpolation,lem-inexact-interpolation})
combined with the contraction of the inner solver. At each scale~$s$ the
method solves a distinct problem $(\mathrm{P}_s)$ that re-discretizes both
the domain and the objective, so each subproblem approximates the
continuous problem rather than a subsampled version of the fine-scale
problem; multiresolution matrix
factorizations~\cite{gillis_multilevel_2012,kondor_multiresolution_2014}
share this re-discretization principle.

Third, a large literature treats the continuous problem
\cref{eq-abstract-continuous-problem} directly, before discretization:
the calculus of variations~\cite{gelfand_calculus_2000}, optimization in
function spaces and PDE-constrained
optimization~\cite{hinze_optimization_2009}, basis and spectral
expansions~\cite{benedetto_wavelets_1993,trefethen_approximation_2019},
parametric families and nonparametric
estimators~\cite{chen_tutorial_2017,yan_learning_2023}, and
convex duality reformulations~\cite{chuna_dual_2025}. In the standard
dichotomy, these are optimize-then-discretize approaches; we deliberately
take the discretize-then-optimize stance, because the problems we target
arrive as fine-grid discretizations and the practical
question is how to solve those discretizations cheaply. Uniformly spaced
grids suffice for Lipschitz functions in both theory and practice. More
sophisticated schemes that adapt the grid to the local variation of
$f$~\cite{stetter_analysis_1973,trefethen_approximation_2019}
are compatible with many of the results below.

Finally, two familiar mechanisms operate inside the method. Each scale
warm starts the next, which is a standard device in numerical
optimization~\cite{adcock_restarts_2025}. The lazy
variant updates only a subset of coordinates at each scale, which connects it
to greedy and block coordinate descent
methods~\cite{dhillon_nearest_2011,nesterov_efficiency_2012,xu_BlockCoordinateDescent_2013}.

\subsection{Reproducible research} The data files and scripts used to generate the numerical results presented in this paper are available from the GitHub repository~\cite{Richardson_multiscale_paper_code}.

\section{The multiscale optimization
method}\label{the-multiscale-optimization-method}

This section constructs the method and establishes that, at every scale, the method solves a consistent finite-dimensional discretization of the same continuous problem~\eqref{eq-abstract-continuous-problem} (same structure with rescaled data) so anything known about the fine-scale problem transfers to the coarser scales. We introduce the construction through a concrete example, then define he transfer operators between scales, the constraint rescaling rules, and the greedy and lazy algorithm variants.

\subsection{Motivating example}\label{sec-motivating-example}

Consider recovering a smooth probability density function from polynomial measurements. We formulate this as an inverse least-squares problem over differentiable densities $f$ on $\mathcal{D}=[-1,1]$:
\[
  \min_{f} \Set{ \half\left\lVert \mathcal{A}(f) - y \right\rVert_2^2 + \half\lambda \left\lVert f'\right\rVert_2^2 | \left\lVert f \right\rVert_1 = 1
  \text{ and } f \geq 0},
\]
where the constraints ensure $f$ is a valid probability density.
The measurement operator~\(\mathcal{A}\) maps $f$ to an $M$-vector with elements
\begin{equation*}
  \mathcal{A}(f)[m] = \left\langle a_m, f \right\rangle \quad \text{for}\quad m=1,\dots,M,
\end{equation*}
with normalized Legendre polynomials
\(
  a_m(t) = (\tfrac{2m+1}{2})^{1/2} \sum_{k=0}^m \binom{m}{k}\binom{m+k}{k} \left(\frac{t-1}{2}\right)^k.
\)
The regularization term
$\left\lVert f'\right\rVert_2^2 = \int_{-1}^1 (f'(t))^2\, dt$
encourages smoothness in the recovered density.
We discretize this continuous problem on a uniform grid to obtain
\begin{equation}\label{eq-motivating-example-finest}
\min_{x_1}\Set{ \half\left\lVert A_1 x_1 - y \right\rVert_2^2 + \half\lambda x_1^\top G_1 x_1 | \lVert x_1 \rVert _{1} =1 \text{ and } x_1 \geq 0},
\end{equation}
where the vector \(x_1=(x_1{[1]},x_1{[2]},\dots,x_1{[I_1]})\) represents the uniform discretization of \(f\) on \([-1,1]\) at the finest scale with grid spacing $\Delta t_1 = 2/(I_1-1)$:
\[
x_1{[i]}=f(t_1{[i]})\Delta t_1, \qquad \text{with} \qquad t_1[i] = -1 + 2\tfrac{i - 1}{I_1-1}.
\]
The discretized measurement operator $A_1$ and graph Laplacian $G_1$, which approximates the second derivative operator, are
\[
 A_1[m, i] = a_m(t_1[i])
 \quad \text{and} \quad
 G_1 = \frac{1}{\Delta t_1^3}
 \scalemath{0.75}{
 \begin{bmatrix}
  1 & -1 & & & \\
  -1 & 2 & -1 & & \\
  & \ddots & \ddots & \ddots &\\
  & & -1 & 2 & -1\\
  & & &  -1 & 1
\end{bmatrix}}.
\]
We return to this example in \cref{sec-method-overview}, after introducing
the general multiscale framework, and report its numerical behavior in
\cref{sec-motivating-example-numerics}.

\subsection{Method overview}\label{sec-method-overview}

We solve the continuous optimization problem
\cref{eq-abstract-continuous-problem} via
successive grid refinement. Starting from a coarse discretization, we solve the discretized problem at each scale, interpolate the solution to the next finer grid, and use it to warm-start the optimization. \Cref{alg-multiscale} formalizes this refinement process.

We index discretization levels by a \emph{scale} parameter $s$, where $s=1$ denotes the finest discretization and increasing $s$ corresponds to progressively coarser grids. At scale $s$, the discretized problem becomes
\begin{equation} \label{eq-scale-s-discretized-problem}
  \min_{x_s}\Set{\tilde{\mathcal{L}}_s(x_s) | x_s \in \tilde{\mathcal{C}}_s}, \tag{\text{$\mathrm{P_s}$}}
\end{equation}
where $x_s\in\mathbb{R}^{I_s}$, and $\tilde{\mathcal{L}}_s$ and
$\tilde{\mathcal{C}}_s$ denote the scale-$s$ discretizations of the
objective and constraint set. The set $\tilde{\mathcal{C}}_s\subseteq
\mathbb{R}^{I_s}$ imposes the continuous constraint on the $I_s$ grid
values and is typically a continuum, such as the rescaled simplex in the example
below, rather than a discrete set. \emph{Discretized} refers to the
finite number of variables, not to the structure of the set.
\Cref{sec-constraints-with-multiscale} gives the constructions.

We return to the motivating example from \cref{sec-motivating-example} to demonstrate how the general framework applies. Recall that we recover a probability density from Legendre polynomial measurements, discretized at the finest scale as \cref{eq-motivating-example-finest}. At scale \(s\), the discretized problem becomes
\begin{equation}\label{eq-motivating-example-scale-s}
\min_{x_s}\Set{ \half\left\lVert A_s x_s - y \right\rVert_2^2 + \half\lambda x_s^\top G_s x_s | \lVert x_s \rVert _{1} = 2^{1-s} \text{ and } x_s \geq 0}
\end{equation}
where \(x_s\) is an \(I_s = 2^{s-1} + 1\) vector that discretizes $f$ at scale $s$ with grid spacing $\Delta t_s = 2^{s-1}\Delta t_1$. The constraint $\lVert x_s \rVert_1 = 2^{s-1}$ scales with the grid spacing to preserve the continuous constraint $\int f(t)\, dt = 1$ across all scales; see \cref{sec-constraints-with-multiscale}. The scaled measurement operator $A_s$ and graph Laplacian $G_s$ are defined as
\[
A_s = 2^{s-1}A_1[:, 1:2^{s-1}:I_1]
\quad \text{and} \quad
G_s = 2^{1-s}G_1[1:2^{s-1}:I_1, 1:2^{s-1}:I_1],
\]
where the slicing notation $1:2^{s-1}:I_1$ selects every $(2^{s-1})$-th point from the finest grid. The prefactors $2^{s-1}$ and $2^{1-s}$ ensure the two main quantities in the loss remain balanced at each scale:
\[
A_s x_s \approx A_1 x_1 \quad \text{and}\quad x_s^\top G_s x_s \approx x_1^\top G_1 x_1.
\]
This balancing ensures each scaled problem is regularized similarly by \(\lambda\).

\subsection{Discretization, coarsening, and interpolation}

Unless otherwise stated, we consider scalar functions \(f:[0,1]\to\mathbb{R}\) sampled on the uniform grid
\[
x_s[i] = f(t_s[i]), \qquad \text{with} \qquad t_s[i] = \tfrac{i - 1}{I_s-1};
\]
as discussed in \cref{rem-scope}, results for a general compact interval \([\ell,u]\) follow by an affine change of variables, with \(L_f\) replaced by the effective constant \((u-\ell)L_f\).

The multiscale framework requires operators to transfer solutions of \cref{eq-scale-s-discretized-problem} between scales and to eventually construct approximate solutions to the continuous problem \cref{eq-abstract-continuous-problem}.
Any coarsening and
interpolation method could be used, including splines~\cite{de_boor_practical_2001} and Chebyshev and Lagrange polynomials~\cite{trefethen_approximation_2019}. We focus on dyadic coarsening, which keeps every second point, and midpoint linear interpolation. These are defined below.
\Cref{fig-discretization-example} illustrates this discretization hierarchy and midpoint interpolation scheme.

\begin{definition}[Dyadic Coarsening]
  \protect\hypertarget{def-subsampled-coarsening}{}\label[definition]{def-subsampled-coarsening}
  Dyadic coarsening maps a vector \(x\in\mathbb{R}^I\)
  to \(\coarse{x}\in\mathbb{R}^{\lfloor (I+1)/2 \rfloor}\)
  by selecting odd-indexed entries:
  \[
  \coarse{x}[i] = x[2i-1],
  \qquad i=1,\ldots,\lfloor (I+1)/2 \rfloor.
  \]
\end{definition}

\begin{definition}[Midpoint Linear
Interpolation]\protect\hypertarget{def-midpoint-linear-interpolation}{}\label[definition]{def-midpoint-linear-interpolation}
Midpoint linear interpolation inserts a linearly interpolated point between each adjacent pair of entries in \(x\in\mathbb{R}^I\) to produce \(\interp{x}\in\mathbb{R}^{2I-1}\):
\[
\interp{x}[i]=\begin{cases}
x[\tfrac{i+1}{2}]   & \text{if $i$ is odd,}\\
\tfrac{1}{2}\left( x[\tfrac{i}{2}] + x[\tfrac{i}{2}+1] \right) & \text{if $i$ is even.}
\end{cases}
\]
\end{definition}

The operators in the coarsening and interpolation definitions translate between scales \(s\) with the notation
\[
\coarse{x}_s = x_{s+1} \qquad \text{and} \qquad \interp{x}_s = x_{s-1}.
\]
The overline operator $\coarse{(\cdot)}$ maps \(x_s\) to the coarser scale $s+1$; the underline operator $\interp{(\cdot)}$ maps to the finer scale $s-1$; see \cref{fig-discretization-example}.

We refer to the vector of newly added interpolated points as the \emph{free variables} $x_{(s)}$, a term justified in \cref{sec-two-multiscale-algorithms}.

\begin{definition}[Vector of Free
Variables]\protect\hypertarget{def-free-variables}{}\label[definition]{def-free-variables}
Given the $(I_{s+1})$-vector $x_{s+1}$ and its interpolation $\interp{x}_{s+1}=x_s$, the $(I_{s+1}-1)$-vector of \emph{free variables} has entries
\[
\textstyle
x_{(s)}[i] = \frac{1}{2}( x_{s+1}[i] + x_{s+1}[i+1]), \quad i=1,\dots,I_{s+1}-1.
\]
These are the newly added points in $x_s$ when interpolating $x_{s+1}$.
At the coarsest scale $s=S$, we set $x_{(S)}=x_{S}$.
\end{definition}

\begin{figure}[t]

\centering

\includegraphics[width=0.75\linewidth]
{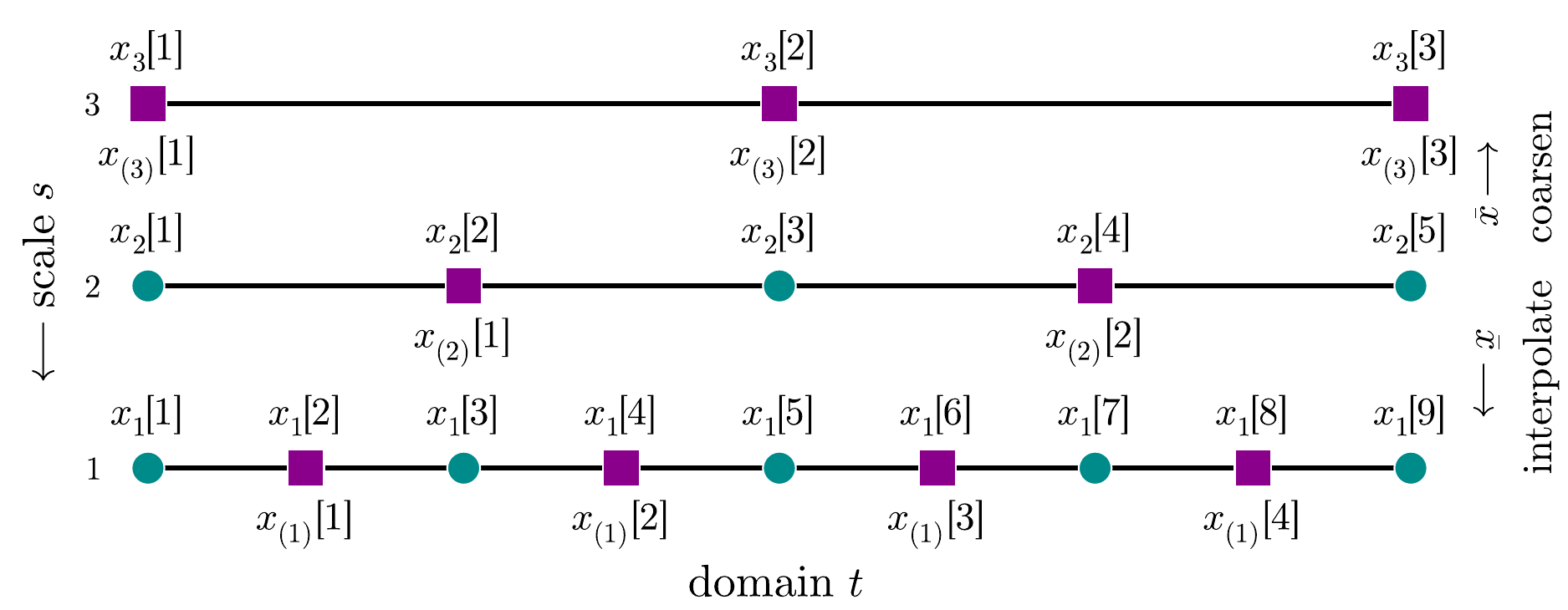}

\caption{Discretization of an interval with \(S=3\) scales. Each scale \(s\) has
\(2^{S-s+1}+1\) points. Starting at scale \(s=3\), newly added points at each scale (the free variables \(x_{(s)}[i]\)) are shown in purple squares, while existing points are shown in teal circles.
}\label{fig-discretization-example}

\end{figure}

To convert a discrete solution \(x_1\) of
\cref{eq-scale-s-discretized-problem} at the finest scale $s=1$ to an
approximate solution to the continuous problem
\cref{eq-abstract-continuous-problem}, we define a
piecewise linear approximation. This approximation
stitches secants through
(possibly approximate) sample points~\cite{de_boor_practical_2001}.%

\begin{definition}[Piecewise Linear
Approximation]\protect\hypertarget{def-piecewise-linear-approximation}{}\label[definition]{def-piecewise-linear-approximation}
Let the $I$-vector \(x\) represent (possibly approximate)
samples of a function \(f\) at locations \(t_1 < \dots < t_I\). The
piecewise linear approximation constructed from \(x\) is
\[
\hat{f}_x(t) =
\tfrac{{x[i+1] - x[i]}}{t_{i+1} - t_{i}}( t - t_i)+ x[i] \quad \text{where}\quad t_{i}\leq t\leq t_{i+1}\quad \text{for} \quad i = 1,\dots,I.
\]
\end{definition}

\begin{remark}
The subscript in $\hat{f}_x$ identifies the vector $x$ that defines the approximation.
When $x$ contains samples from a function, that is, $x[i] = f(t_i)$, we omit the subscript and write \(\hat{f}(t)\).
\end{remark}

\subsection{Discretizing constraints across scales}\label{sec-constraints-with-multiscale}

The multiscale approach requires discretizing
\cref{eq-abstract-continuous-problem} at multiple scales. Given a
discretization of the continuous constraint \(\mathcal{C}\) from~\eqref{eq-abstract-continuous-problem} to the finest scale \(\tilde{\mathcal{C}}_1\) in~\eqref{eq-scale-s-discretized-problem}, we must determine the appropriate
discretized constraint sequence \(\tilde{\mathcal{C}}_s\) at coarser scales $s=S-1, S-2, \dots, 1$.

Pointwise constraints \(\mathcal{C} = \left\{f:\mathcal{D}\to\mathbb{R} \mid f(t)\in\mathcal{X}\right\}\)
transfer verbatim to any scale:
\(\tilde{\mathcal{C}}_s=\left\{x_s\in\mathbb{R}^{I_s} \mid x[i]\in\mathcal{X}\right\}.\)
However, linear and moment constraints require careful scaling to maintain consistency across discretization levels.

Consider a \(p\)-norm constraint on \(f\):
\[
\mathcal{C}=\Set{f:\mathcal{D}\to\mathbb{R} | \left\lVert f\right\rVert _p = c}.
\]
The natural discretization
\(
\tilde{\mathcal{C}}_1=\{x_1\in\mathbb{R}^{I_1} | \left\lVert x_1\right\rVert _p = c\}
\)
correctly approximates $\mathcal{C}$ when we scale the entries of \(x_1\) by the grid size:
\(
x_1[i] = f(t_1[i]) (\Delta t)^{1/p}.
\)
With this scaling, the discrete norm approximates the continuous norm:
\[
c^p = \left\lVert x_1\right\rVert _p^p = \textstyle\sum_{i=1}^{I_1} f(t_1[i])^p \Delta t \approx \int_{\mathcal{D}} f(t)^p\, dt = \left\lVert f\right\rVert _p^p,
\]
as \(I_1\to \infty\).
At coarse scale $s$, we modify the constraint to
\[
\tilde{\mathcal{C}}_s=\Set{x_s\in\mathbb{R}^{I_s} | \left\lVert x_s\right\rVert _p
= c \cdot ({I_s}/{I_1})^{1/p} }.
\]
This scaling preserves the constraint across scales. The Lipschitz property of \(f\)
ensures neighboring entries in \(x_s\) remain similar, so
\[
   \textstyle c^{p} = \frac{I_1}{I_s}\left\lVert x_s\right\rVert_p^p =\frac{I_1}{I_s}\sum_{i=1}^{I_s} (x_s[i])^p  \approx \sum_{i=1}^{I_1} (x_1[i])^p = \left\lVert
  x_1\right\rVert_p^p.
\]

Linear constraints require analogous scaling. For \(\lipschitz{g_k}\)-Lipschitz functions \(g_k:\mathcal{D}\to\mathbb{R}\), the continuous constraint
\[
  \mathcal{C}
  =
  \Set{f:\mathcal{D}\to\mathbb{R} |\left\langle g_k, f\right\rangle = b[k],\, k\in[K]} \quad \text{where} \quad b\in\mathbb{R}^K,
\]
discretizes at scale \(s\) as
\[
  \tilde{\mathcal{C}}_s=\Set{x_s\in\mathbb{R}^{I_s} | A_s x_s = b\cdot({I_s}/{I_1}) },
\]
where
\(
  A_s[k,i]=g_k(t_s[i])(\Delta t)^{1/2} \ \text{and} \ x_s[i] = f(t_s[i])
  (\Delta t)^{1/2}.
\)

We bound the discretization error to justify our constraint modification across scales.
To do this, we extend of the definition of Lipschitz
continuity to vectors.
\begin{definition}[Lipschitz
Vector]\protect\hypertarget{def-lipschitz-vector}{}\label[definition]{def-lipschitz-vector}
A vector \(x \in \mathbb{R}^I\) is \(\lipschitz{x}\)-Lipschitz when,
\[
\left\lvert x[i] - x[j] \right\rvert \leq \lipschitz{x} \left\lvert i - j \right\rvert\qquad \text{for all} \qquad i,j\in[I].
\]

\end{definition}
We also get a correspondence between Lipschitz functions and vectors.
\begin{corollary}[Lipschitz Vectors and
Functions]\protect\hypertarget{cor-lipschitz-vectors-and-functions}{}\label[corollary]{cor-lipschitz-vectors-and-functions}
Given \(I\) samples \(x[i]=f(t[i])\) of an \(\lipschitz{f}\)-Lipschitz
function \(f:\mathbb{R}\to \mathbb{R}\) on a uniformly spaced grid
\(\{t[i]\}\), the vector \(x \in \mathbb{R}^I\) is
\(\lipschitz{x}=\lipschitz{f}\Delta t\) Lipschitz.

\end{corollary}

\begin{proof}
Let \(i,j\in[I]\). Our grid has points
\(t[i] = (i-1)\Delta t\), so we have
\begin{align*}
\left\lvert x[i]-x[j] \right\rvert &=\left\lvert f(t[i])-f(t[j]) \right\rvert \leq \lipschitz{f} \left\lvert t[i] - t[j] \right\rvert = \lipschitz{f} \Delta t \left\lvert i-j \right\rvert.
\endproofhere
\end{align*}
\end{proof}

The following proposition quantifies how linear constraints transform under coarsening.
\begin{proposition}[Single Linear Constraint
Scaling]\protect\hypertarget{prp-linear-constraint-scaling}{}\label[proposition]{prp-linear-constraint-scaling}
Let \(a\in\mathbb{R}^I\) discretize a bounded
\(\lipschitz{g}\)-Lipschitz function \(g:[0, 1]\to\mathbb{R}\), and let $x$ discretize a bounded $\lipschitz{f}$-Lipschitz function $f$. Define the Lipschitz constant for the product function $fg$ as
\(\lipschitz{fg}=\lipschitz{f} \left\lVert g \right\rVert_{\infty} + \lipschitz{g} \left\lVert f \right\rVert_{\infty}\).
If \(\langle a, x \rangle = b\), then
\begin{alignat*}{2}
  \left\lvert \left\langle \coarse{a}, \coarse{x} \right\rangle - \half b  \right\rvert
  &\leq \tfrac14\lipschitz{fg}  \tfrac{I}{I-1}
  &&\quad\text{for even $I$},
  \\
  \left\lvert \left\langle \coarse{a}, \coarse{x} \right\rangle - \half\tfrac{I+1}{I}b \right\rvert
  &\leq \half\lipschitz{fg}
  &&\quad\text{for odd $I$}.
\end{alignat*}
\end{proposition}

\begin{proof}
See \cref{sec-linear-constraint-scaling-proof}.
\end{proof}

\Cref{prp-linear-constraint-scaling} extends naturally
to multiple linear constraints, as described in the following corollary.

\begin{corollary}[Matrix Linear Constraint
Scaling]\protect\hypertarget{cor-matrix-linear-constraint}{}\label[corollary]{cor-matrix-linear-constraint}
Let \(g_k:[0, 1]\to\mathbb{R}\) be bounded
\(\lipschitz{g_{k}}\)-Lipschitz functions for \(k\in[K]\). Let
\(A\in\mathbb{R}^{K\times I}\) be the discretization
\(A[k,i] = g_k(t[i])\), and define $\coarse{A}=A[:,\texttt{begin}\mathbin{:}2\mathbin{:}\texttt{end}]$ as the $(K\times \left\lceil I\right\rceil_e/2)$-submatrix with every other column removed, where
\(\left\lceil I\right\rceil_e\) rounds up to the nearest even integer. Define $\lipschitz{fg} = \max_{k\in[K]}(\lipschitz{fg_k}) = \max_{k\in[K]}(\lipschitz{f} \left\lVert g_k \right\rVert_{\infty} + \lipschitz{g_k} \left\lVert f \right\rVert_{\infty})$ as the maximum Lipschitz constant of the product functions $\{fg_k\}_{k\in[K]}$.
If \(Ax=b\in\mathbb{R}^K\), then
\begin{alignat*}{2}
\left\lVert  \coarse{A}\coarse{x} - \half b \right\rVert _2
&\leq \sqrt{K}\cdot \tfrac14\lipschitz{fg}  \tfrac{I}{I-1}
&&\quad\text{for even $I$},
\\
\left\lVert  \coarse{A}\coarse{x} - \tfrac{I+1}{I}\half b \right\rVert _2
&\leq \sqrt{K}\cdot \half\lipschitz{fg} 
&&\quad\text{for odd $I$}.
\end{alignat*}
\end{corollary}

\begin{proof}
Let \(a_k = A[k,:]\in\mathbb{R}^{I}\) be the \(k\)th row vector in
\(A\). By \cref{prp-linear-constraint-scaling},
\begin{align*}
\textstyle
\left\lVert  \coarse{A}\coarse{x} - \frac{b}{2} \right\rVert _2^2 &= \textstyle\sum_{k=1}^K  \left\lvert \left\langle \coarse{a}_k, \coarse{x} \right\rangle - \frac{b[k]}{2} \right\rvert^2 \leq \textstyle\sum_{k=1}^K \left(\tfrac{I}{I-1} \tfrac{\lipschitz{fg_k} }{4}\right)^2 \\
&\leq \left(\tfrac{I}{I-1} \tfrac{1}{4}\right)^2 \textstyle\sum_{k=1}^K \max_k\{(\lipschitz{fg_k})^2\} = \left(\tfrac{I}{I-1} \tfrac{1}{4}\right)^2 (\lipschitz{fg})^2 K,
\end{align*}
for even \(I\).
Taking square roots completes this case, and the case for odd \(I\) is
similar.
\end{proof}

The 1-norm constraint admits analogous scaling bounds.

\begin{proposition}[L1-norm Constraint
Scaling]\protect\hypertarget{prp-l1-norm-constraint-scaling}{}\label[proposition]{prp-l1-norm-constraint-scaling}
Let $x\in\mathbb{R}^I$ discretize a bounded $\lipschitz{f}$-Lipschitz function $f$.  If \(\lVert x \rVert_1 = b\), then the coarsened constraint satisfies the bounds
\begin{alignat*}{2}
  \left\lvert \left\lVert \coarse{x} \right\rVert_1 - \half b \right\rvert &\leq \tfrac14\lipschitz{f}  \tfrac{I}{I-1}
  &&\quad\text{for even $I$},
  \\
  \left\lvert \left\lVert \coarse{x} \right\rVert_1 - \tfrac{I+1}{I}\half b \right\rvert &\leq \half\lipschitz{f}
  &&\quad\text{for odd $I$}.
\end{alignat*}
\end{proposition}

\begin{proof}
The proof is similar to the proof of \cref{prp-linear-constraint-scaling} (see \cref{sec-linear-constraint-scaling-proof})
by considering the vector \(a=\mathrm{sign}(x)\), and using the
Lipschitz constant \(\lipschitz{f}\) for \(f\) to bound
\(\left\lvert x[i] - x[i+1] \right\rvert\) in place of
\(\lipschitz{fg}\).
\end{proof}

Proposition \ref{prp-l1-norm-constraint-scaling} specializes to a clean asymptotic behavior in the normalized case. For \(z = x / \lVert x \rVert_1\), the bound for even \(I\) becomes
\[
  \left\lvert \left\lVert \coarse{z} \right\rVert_1 - \half\right\rvert \leq \tfrac1{4b}\lipschitz{f}  \tfrac{I}{I-1}.
\]
As the discretization refines (\(I\to\infty\)), we have \(b=\lVert x \rVert_1\to\infty\), so the bound vanishes and \(\left\lVert\coarse{z} \right\rVert_1 \to 1/2\). The same limit holds for odd \(I\), and also extends to other \(p\)-norms.

\subsection{Greedy and lazy multiscale algorithms}\label{sec-two-multiscale-algorithms}

We present two variants of the multiscale method (\cref{alg-multiscale}) outlined in \cref{sec-method-overview} for solving \cref{eq-abstract-continuous-problem}: the \emph{greedy} variant, which optimizes all grid points at each scale, and the \emph{lazy} variant, which optimizes only the newly interpolated points at each finer scale. Both variants apply an update rule \(x^{k+1}=U(x^k)\) at each scale, and our convergence analysis holds for any such update rule that achieves \(q\)-linear iterate convergence in the following sense.

\begin{definition}[\(q\)-linear iterate convergence]
  \protect\hypertarget{def-iterate-convergence}{}\label[definition]{def-iterate-convergence}
An iterative algorithm with update rule \(x^{k+1}=U(x^k)\) has global q-linear iterate
convergence with rate \(q\in[0,1)\) when, for any initialization $x^0\in\tilde{\mathcal{C}}$, there is a minimizer $x^*$ of $\tilde{\mathcal{L}}$ over $\tilde{\mathcal{C}}$ (which may depend on \(x^0\)) such that each update contracts the distance to $x^*$ by a factor of $q$:
\begin{equation}\label{eq-q-linear}
\left\lVert x^{k+1}-x^* \right\rVert _2 \leq q \left\lVert x^{k} -x^* \right\rVert _2  \quad\text{for all $k\geq0$}.
\end{equation}
Consequently, $\|x^K - x^*\|_2\le q^K\|x^0-x^*\|$ after $K$ updates. 
\end{definition}

Our software implementation (\cref{sec-numerical-experiments}) uses projected gradient descent as the update rule, described by
\begin{equation}\label{def-projected-gradient-decent}
x^{k+1} \leftarrow U(x^k) = P_{\tilde{\mathcal{C}}}\left( x^k - \alpha\nabla\tilde{\mathcal{L}}(x^k) \right)
\end{equation}
where
\(
P_{\tilde{\mathcal{C}}}(x) = \mathop{\rm argmin} \Set{\norm{x - y}_2 | y\in\tilde{\mathcal{C}}}
\)
denotes the projection of $x$ onto \(\tilde{\mathcal{C}}\); when $\tilde{\mathcal{C}}$ is nonconvex the projection may be non-unique, in which case any selection may be used.
The standard analysis gives the following %
$q$-linear rate for smooth, strongly convex objectives over closed convex constraint sets; the smoothness and strong convexity constants for the discretized loss are notated with $\mathcal{S}_{\tilde{\mathcal{L}}}$ and $\mu_{\tilde{\mathcal{L}}}$.
\begin{lemma}[Descent
Lemma {\cite[Theorem 2.1.15]{nesterov-smooth-2018}}]\protect\hypertarget{lem-descent-lemma}{}\label[lemma]{lem-descent-lemma}
Projected gradient descent with stepsize \(0<\alpha \leq 2/({\smoothness{\tilde{\mathcal{L}}}+\convexity{\tilde{\mathcal{L}}}})\) has iterate
convergence with rate
\[
q(\alpha) = \sqrt{1 - \tfrac{2\alpha \smoothness{\tilde{\mathcal{L}}} \convexity{\tilde{\mathcal{L}}}}{\smoothness{\tilde{\mathcal{L}}} + \convexity{\tilde{\mathcal{L}}}}}.
\]
This is minimized at a stepsize of
\(\alpha = {2}/({\smoothness{\tilde{\mathcal{L}}}+\convexity{\tilde{\mathcal{L}}}})\),
where we obtain the iterate convergence
\(
q(\alpha) = ({c-1})/({c+1})
\)
for a \emph{condition number} of
\(c=\smoothness{\tilde{\mathcal{L}}} /\convexity{\tilde{\mathcal{L}}}\).
\end{lemma}

Projected gradient descent over a convex set is thus one update rule that satisfies \cref{def-iterate-convergence}; the analysis below uses only the definition, never the specific iterate update. Although we prescribe the number of iterations $K_S$ at each scale in advance, in prace we may also additional convergence criteria (such as a small gradient) at each scale to terminate early. 

\begin{algorithm}[t]
  \caption{The Multiscale Method: step 6 specifies Greedy or Lazy versions}
\label{alg-multiscale}\label{alg-lazy-multiscale}\label{alg-greedy-multiscale}
  \begin{algorithmic}[1]

\STATE{
  Randomly initialize \(x_s^0\) at the coarsest scale \(s=S\).}
\STATE{
  Perform \(K_s\) iterations of a $q$-linear update $U(x)$ \cref{eq-q-linear} on \(x_{s}^0\) to approximately solve
  \cref{eq-scale-s-discretized-problem} at scale \(s=S\) and obtain
  \(x_s^{K_s}\).}
\FOR{scales \(s = S-1, S-2, \dots, 1\)}
\STATE{
  Interpolate \(x_{s+1}^{K_{s+1}}\) to obtain \(\interp{x}_{s+1}^{K_{s+1}}\) according
  to \cref{def-midpoint-linear-interpolation}.}
\STATE{
  Initialize the next finest scale \(x_{s}^0=\interp{x}_{s+1}^{K_{s+1}}\)
  using the previous scale's interpolated solution.}
\STATE{
  Perform \(K_s\) iterations of a $q$-linear update $U(x)$ \cref{eq-q-linear} on \(x_s^0\) (Greedy) or \(x_{(s)}^0\) (Lazy) to approximately solve
  \cref{eq-scale-s-discretized-problem} at scale \(s\) and obtain \(x_s^{K_s}\).
}
  \ENDFOR
\RETURN Approximate solutions \(x_1^{K_1}\) and \(\hat{f}_{x_1^{K_1}}\) for \cref{eq-scale-s-discretized-problem} at \(s=1\) and \cref{eq-abstract-continuous-problem} respectively.

  \end{algorithmic}
\end{algorithm}

\section{Convergence to the fine-scale solution}\label{sec-multiscale-convergence}

Both the standard and multiscale approaches converge because, by design, they both apply a convergent algorithm at the finest scale. The multiscale approach converges faster because coarse-scale solutions warm start the fine-scale iterations. This section formalizes this claim. First, interpolation-error lemmas (\cref{sec-analysis-lemmas}) bound the error injected when a coarse solution is carried to the next finer grid: midpoint interpolation of exact values of an $L_f$-Lipschitz function has error at most $\half L_f/\sqrt{I-1}$ (\cref{lem-exact-interpolation}), and interpolating \emph{inexact} values inflates that error by at most $\sqrt2$ while adding the same interpolation term (\cref{lem-inexact-interpolation}).
Chaining the inner-solver contraction (\cref{def-iterate-convergence}) with interpolation from the coarest scale to the finest scale telescopes into a closed-form bound for both greedy (\cref{sec-greedy-multiscale}) and lazy (\cref{sec-lazy-multiscale}) multiscale methods.

\subsection{Interpolation error for Lipschitz functions}\label{sec-analysis-lemmas}

To analyze multiscale method convergence, we first establish bounds on discretizing Lipschitz functions and approximating them with linear interpolations. Throughout, we solve the discretized problem
\cref{eq-scale-s-discretized-problem} in one dimension with any
algorithm that achieves iterate convergence at the rate specified in \cref{def-iterate-convergence}.

The following lemma provides the foundation for analyzing linear interpolations of Lipschitz functions and, ultimately, our multiscale algorithm.

\begin{lemma}[Lipschitz Function
Interpolation]\protect\hypertarget{lem-lipschitz-interpolation}{}\label[lemma]{lem-lipschitz-interpolation}
Let \(f:\mathbb{R}^N \to \mathbb{R}\) be \(\lipschitz{f}\)-Lipschitz.
For all \(a,b\in \mathbb{R}^N\) and
\(\lambda_1,\lambda_2\ge0\), $\lambda_1+\lambda_2=1$, the error between \(f\) evaluated at a convex combination of \(a\) and \(b\) and the linear interpolation of $f(a)$ and $f(b)$ satisfies
\begin{equation}\phantomsection\label{eq-lipschitz-interpolation-bound}{
\left\lvert f(\lambda_1 a + \lambda_2 b) - (\lambda_1 f(a)+\lambda_2 f(b)) \right\rvert\leq 2\lipschitz{f} \lambda_1\lambda_2\left\lVert a-b \right\rVert _2.
}\end{equation}
\end{lemma}

\begin{proof}
See \cref{sec-lipschitz-interpolation-proof} and \cref{fig-lipschitz-interpolation}.
\end{proof}

\begin{figure}[t]
\centering
\includegraphics[width=0.75\linewidth,height=\textheight,keepaspectratio]{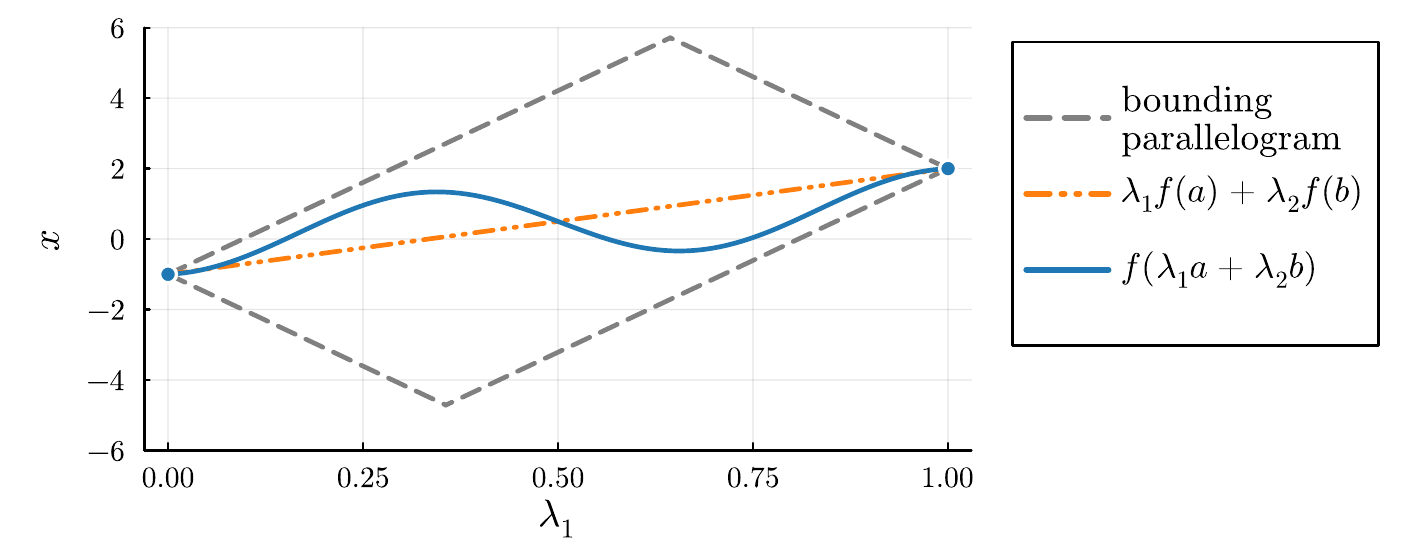}
\caption{Example for
\cref{lem-lipschitz-interpolation}\label{fig-lipschitz-interpolation} with the function
\(x=f(t) = t - \cos(3\pi t)\) on the interval \([0, 1]\). The
function \(f\) must lie within the dashed parallelogram since
it is \(\lipschitz{f}=1+3\pi\) Lipschitz.
\cref{lem-lipschitz-interpolation} uses the parallelogram constraint to bound the
distance between \(f\) and the linear interpolation
\(\lambda_1 f(a) + \lambda_2 f(b)\).
\Cref{lem-lipschitz-interpolation-tightness} shows that the
bound in \cref{eq-lipschitz-interpolation-bound} is tight for any
given \(\lambda_1,\lambda_2\ge0\) with \(\lambda_1+\lambda_2=1\).}
\end{figure}

Since \cref{lem-lipschitz-interpolation} is central to the
analysis of linear interpolations of Lipschitz functions and ultimately
our multiscale algorithm, we require a tight bound on
the interpolation error. The bound in~\eqref{eq-lipschitz-interpolation-bound}
is optimal among all bounds independent of the endpoint values \(f(a)\) and \(f(b)\), as demonstrated by the
explicit
construction in the proof of
\cref{lem-lipschitz-interpolation-tightness}.

\begin{lemma}[Lipschitz Interpolation
Tightness]\protect\hypertarget{lem-lipschitz-interpolation-tightness}{}\label[lemma]{lem-lipschitz-interpolation-tightness}
Given \(\lambda_1,\lambda_2\ge0\) with \(\lambda_1+\lambda_2=1\), points \(a,b\in\mathbb{R}^N\), and constant
\(\lipschitz{f}>0\), there exists an \(\lipschitz{f}\)-Lipschitz
function \(f:\mathbb{R}^N\to\mathbb{R}\) such that~\cref{eq-lipschitz-interpolation-bound} holds with equality.
\end{lemma}

\begin{proof}
Construct
\(
f(x) = \lipschitz{f}\left\lVert x - (\lambda_1 a + \lambda_2 b)\right\rVert_2.
\)
This function is \(\lipschitz{f}\)-Lipschitz with end-point values
\(f(a) =  \lipschitz{f} \lambda_2 \left\lVert a - b \right\rVert_2\) and
\(f(b)= \lipschitz{f}\lambda_1\left\lVert a - b \right\rVert_2\), and
\(f( \lambda_1 a + \lambda_2 b)=0\). Therefore
\(
\left\lvert \lambda_1 f(a) + \lambda_2 f(b) - f( \lambda_1 a + \lambda_2 b)\right\rvert
=2\lipschitz{f}\lambda_1\lambda_2\left\lVert a - b \right\rVert_2.
\)
\end{proof}

To bound the error introduced when interpolating a coarse-scale solution to a finer scale, we apply \cref{lem-lipschitz-interpolation} repeatedly with centre point interpolation ($\lambda_1=\lambda_2=1/2$). This bounds the error between an exact fine-scale discretization
\(x_s^*[j] = f(t_s[j])\) and the linear interpolation
\(x_s=\interp{x}_{s+1}\) obtained from a coarser discretization
\(x_{s+1}[i]=f(t_{s+1}[i])\). See
\cref{fig-exact-and-inexact-interpolation} (left).

\begin{lemma}[Exact
Interpolation]\protect\hypertarget{lem-exact-interpolation}{}\label[lemma]{lem-exact-interpolation}
Let \(f:[0, 1]\to \mathbb{R}\) be \(\lipschitz{f}\)-Lipschitz continuous. On the interval \([0, 1]\), let \(t_s\) be a fine grid with \(J=2I-1\) points, and \(t_{s+1}\) be a coarse grid with \(I\) points. Discretize
the function exactly at both scales to obtain vectors
\(x^*_s\in\mathbb{R}^J\) and \(x^*_{s+1}\in\mathbb{R}^I\) where
\(x^*_s[j] = f(t_s[j])\) and \(x^*_{s+1}[i] = f(t_{s+1}[i])\).
Let \(x_{s+1}=x^*_{s+1}\in\mathbb{R}^I\) and linearly interpolate to
obtain \(x_s=\interp{x}_{s+1}\in\mathbb{R}^J\); see \cref{def-midpoint-linear-interpolation}.
Then the difference between the interpolated \(x_s\) and exact values
\(x_s^*\) is bounded by
\begin{equation}\phantomsection\label{eq-lem-exact-interpolation}{
   \left\lVert x_s - x_s^* \right\rVert_{2}
   \leq\half\lipschitz{f}/\sqrt{ I-1 }.
}\end{equation}

\end{lemma}

\begin{proof}
See \cref{sec-exact-interpolation-proof}.
\end{proof}

When the coarse discretization contains error, the interpolation analysis extends naturally. Suppose we interpolate not from exact values \(x_{s+1}[i]=x_{s+1}^*[i]=f(t_{s+1}[i])\), but
from \emph{approximate values}
\(x_{s+1}[i] = x_{s+1}^*[i] + \delta [i] = f(t_{s+1}[i]) + \delta [i]\), where $\delta$ is an $I$-vector that represents the error at the coarse scale.
\cref{fig-exact-and-inexact-interpolation} (right) visualizes this scenario.

\begin{lemma}[Inexact
Interpolation]\protect\hypertarget{lem-inexact-interpolation}{}\label[lemma]{lem-inexact-interpolation}
Consider the setup of \cref{lem-exact-interpolation} with one modification:
the coarse discretization \(x_{s+1}=x^*_{s+1}+\delta\) contains error $\delta\in\mathbb{R}^I$ before we interpolate to
obtain \(x_s=\interp{x}_{s+1}\). The difference between the interpolated values  \(x_s\) and exact values
\(x_s^*\) satisfies
\begin{equation}\phantomsection\label{eq-lem-inexact-interpolation}{
  \left\lVert x_s - x_s^* \right\rVert_2
  \leq
  \sqrt{ 2 } \left\lVert \delta \right\rVert_2
   +
  \tfrac12\lipschitz{f}/\sqrt{ I-1 } .
}\end{equation}
\end{lemma}

\begin{proof}
See \cref{sec-inexact-interpolation-proof}.
\end{proof}

This bound decomposes into two terms corresponding to distinct error sources. The first term captures the error from interpolating approximate values
\(x_{s+1}=x^*_{s+1}+\delta\) rather than exact values \(x^*_{s+1}\).
The second term captures the linear interpolation error as in \cref{lem-exact-interpolation}.

\begin{figure}[t]
  \centering
  \includegraphics[width=0.495\linewidth,keepaspectratio]{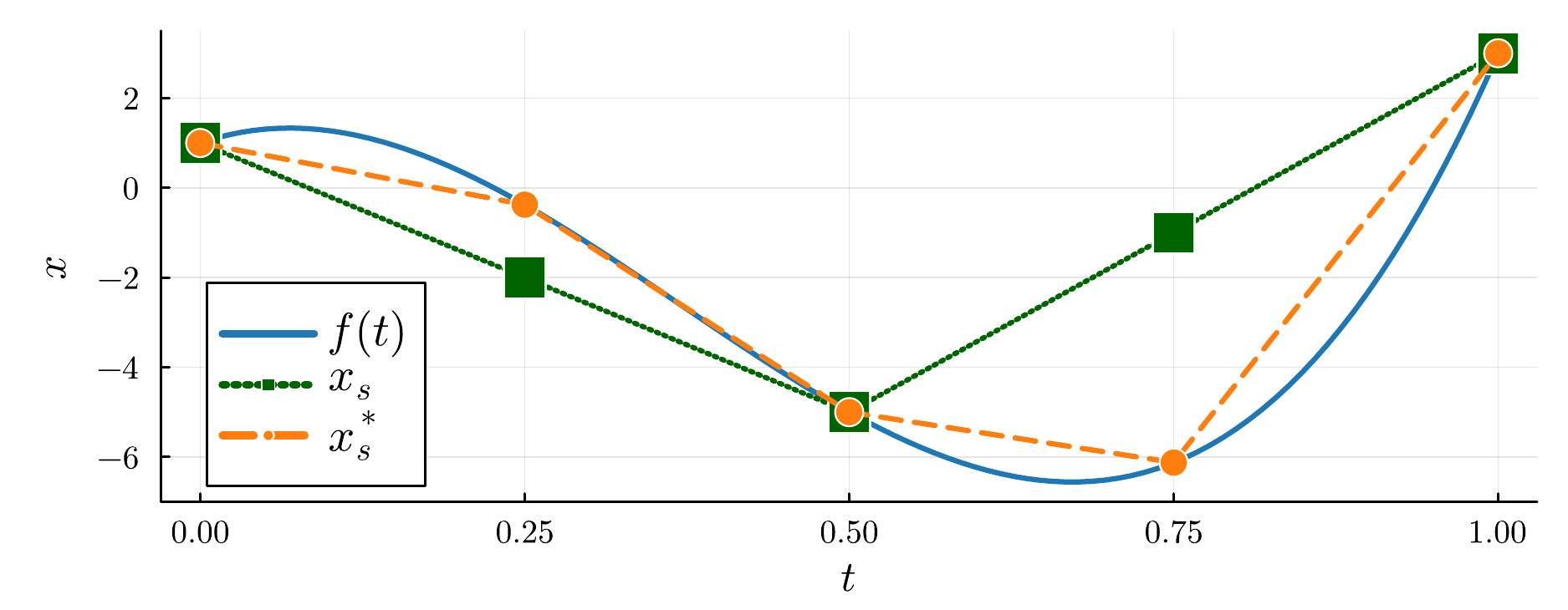}
  \includegraphics[width=0.495\linewidth,keepaspectratio]{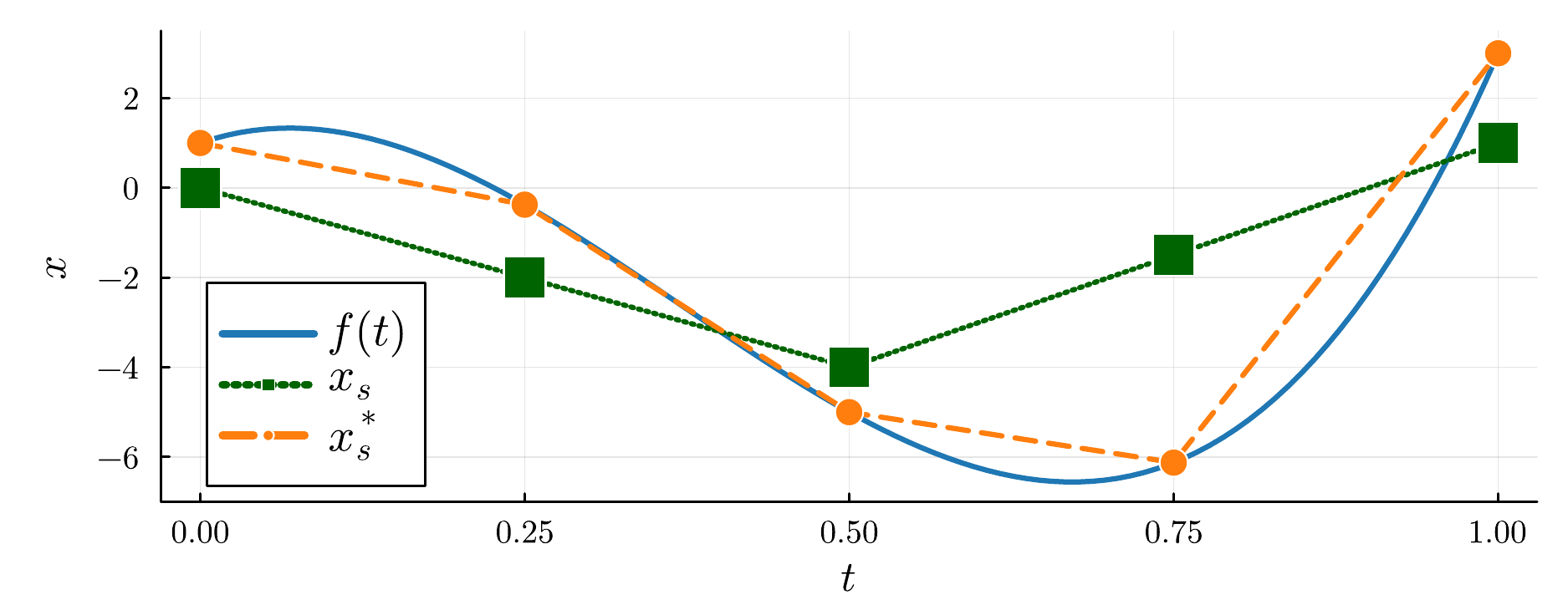}
  \caption{Illustration of
  \cref{lem-exact-interpolation} (left) and
  \cref{lem-inexact-interpolation} (right) for the function
  \(f(t)=72t^3-80t^2+10t+1\). Both lemmas bound the error between
  the true fine-scale discretization \(x^*_s\) (orange circles) and
  the linear interpolation \(x_s=\interp{x}_{s+1}\) (green squares)
  obtained from a coarser-scale discretization. The left figure shows
  \cref{eq-lem-exact-interpolation}, where \(x_{s+1}=x^*_{s+1}\) is exact.
  The right figure shows \cref{eq-lem-inexact-interpolation}, where
  \(x_{s+1}=x^*_{s+1}+\delta\) contains error.
  \label{fig-exact-and-inexact-interpolation}
  }
\end{figure}%

\subsection{Greedy multiscale}\label{sec-greedy-multiscale}

We use the following approach to show convergence of the greedy multiscale algorithm described in \cref{alg-greedy-multiscale}. At each scale \(s\), the
iterate convergence (\cref{def-iterate-convergence}) bounds the error between
the interpolated iterate \(\interp{x}_s^{K_s}\), which becomes the initialization $x^0_{s-1}=\interp{x}^{K_s}_s$ at the next finer scale, and the solution at that finer scale. Recursing from $s=S$ down to $s=1$ yields the following bound.

\begin{theorem}[Greedy Multiscale Error Bound]
\protect\hypertarget{thm-greedy-multiscale-descent-error}{}\label{thm-greedy-multiscale-descent-error}
Let \(q\) denote the rate of convergence of an algorithm
applied to a loss function \(\tilde{\mathcal{L}}\), let
\(\lipschitz{f^*}\) be the Lipschitz constant of the
solution function \(f^*\), and let \(S\) denote the coarsest scale. The greedy multiscale method in \cref{alg-greedy-multiscale} returns iterate \(x_1^{K_1}\) at the finest scale satisfying
\[
  \norm{x^{K_1}_{1} - x^*_{1} }_2 \leq \sqrt{ 2 ^{S-1}}q^{r_S}\left\lVert x_{S}^{0} -x_{S}^* \right\rVert _2 + \tfrac12\lipschitz{f^*}\tfrac{1}{\sqrt{2^{S+1}}}\sum\nolimits_{s=1}^{S-1}  2^{s} q^{r_s},
\]
where \(r_s=\sum_{t=1}^{s} K_{t}\) denotes the cumulative iteration count through scale $s$.
\end{theorem}

\begin{proof}
See \cref{sec-greedy-multiscale-descent-error-proof}.
\end{proof}

\subsection{Lazy multiscale}\label{sec-lazy-multiscale}

We establish convergence of the lazy multiscale algorithm using the framework from \cref{thm-greedy-multiscale-descent-error}. The key modification is to freeze the interpolated values $\interp{x}_s^{K_s}$ that correspond to points from the previous iteration $x_s^{K_s}$. The algorithm updates only the newly interpolated values, denoted $x_{(s)}^k$; see \cref{def-free-variables}.

\begin{theorem}[Lazy Multiscale Error Bound]
\protect\hypertarget{thm-lazy-multiscale-descent-error}{}\label{thm-lazy-multiscale-descent-error}
Assume the same setup as \cref{thm-greedy-multiscale-descent-error}. The lazy multiscale method in \cref{alg-lazy-multiscale} returns iterate \(x_1^{K_1}\) satisfying
\[
\begin{aligned}
\left\lVert x_{1}^{K_{1}} - x_{1}^{*} \right\rVert
\leq \left\lVert x_{S}^{0} - x_{S}^{*} \right\rVert d(K_{S})p(S) +
\half\lipschitz{f^*}\sum\nolimits_{s=1}^{S-1} \frac{1}{\sqrt{ 2^{S-s} }}d(K_{s})p(s),
\end{aligned}
\]
where \(d(K) = q^K \in (0,1)\) and \(p(s) = \prod_{j=1}^{s-1}(1+q^{K_{j}})\). When using constant iterations \(K_s=K\) at each scale, this
reduces to
\[
\begin{aligned}
\left\lVert x_{1}^{K_{1}} - x_{1}^{*} \right\rVert &\leq \left\lVert x_{S}^{0} - x_{S}^{*} \right\rVert d(K)(d(K)+1)^{S-1}  \\
&\qquad+\
\half\lipschitz{f^*} d(K)\frac{ \sqrt{2^S} \left( d(K) + 1 \right)^S - \sqrt{2} \left( d(K) + 1 \right) }{\sqrt{2^S} \left( d(K) + 1 \right) \left( \sqrt{2} d(K) + \sqrt{2} - 1 \right)}.
\end{aligned}
\]

\end{theorem}

\begin{proof}
See \cref{sec-lazy-multiscale-descent-error-proof}.
\end{proof}

We summarize the convergence with
\cref{cor-multiscale-convergence}.

\begin{corollary}[Multiscale
Convergence]\protect\hypertarget{cor-multiscale-convergence}{}\label[corollary]{cor-multiscale-convergence}
Suppose the number of iterations \(K_1\) at the finest scale grows unbounded \(K_1  \to  \infty\). Then under the settings described by \cref{thm-greedy-multiscale-descent-error,thm-lazy-multiscale-descent-error}, both multiscale algorithms (\cref{alg-greedy-multiscale,alg-lazy-multiscale}) converge to a solution \(x_1^*\) that solves \cref{eq-scale-s-discretized-problem} at the finest
scale \(s=1\).

\end{corollary}

\section{Relation between the discretized and continuous
problems}\label{sec-relation-between-the-discretized-and-continuous-problems}

The previous section controls the error the fine-scale solution. 
This section relates solutions of the discretized and continuous problems
\cref{eq-scale-s-discretized-problem} and
\cref{eq-abstract-continuous-problem}. We show that a small error
$\|x_{1}^{K_{1}} - x^* \|_2$
in the fine-scale discretized problem solution implies a small error
$\|\hat{f}_{x_{1}^{K_{1}}} - f^* \|_2$
for the continuous problem, where $f^*$ solves \cref{eq-abstract-continuous-problem} and \(\hat{f}_{x_{1}^{K_{1}}}(t)\) denotes the piecewise linear function constructed from \(x_{1}^{K_{1}}\) according to \cref{def-piecewise-linear-approximation}.
We assume \(t\in\mathbb{R}^I\) forms a uniform grid on the interval
\([0, 1]\) with spacing
\(\Delta t = t[{i+1}] - t[i] = \frac{1}{I-1}\).

\begin{lemma}[Piecewise Linear Function Distance]\protect\hypertarget{lem-piecewise-function-distance}{}\label[lemma]{lem-piecewise-function-distance}
Let \(f,g:\mathbb{R}\to\mathbb{R}\), and let \(x,y\in\mathbb{R}^I\) be samples \(x[i] = f(t[i])\) and \(y[i] = g(t[i])\). Construct piecewise linear approximations (\cref{def-piecewise-linear-approximation}) \(\hat{f}\) of \(f\) and \(\hat{g}\) of \(g\). Then
\[
\|\hat{f} - \hat{g} \|_2^2 \leq \Delta t \|x-y\|_2^2.
\]
\end{lemma}

\begin{proof}
See \cref{sec-piecewise-function-distance-proof}.
\end{proof}

We bound the distance between a Lipschitz
function and its piecewise linear approximation.

\begin{lemma}[Piecewise Linear Function
Approximation]\protect\hypertarget{lem-piecewise-function-approximation}{}\label[lemma]{lem-piecewise-function-approximation}
Let $f$ be an $\lipschitz{f}$-Lipschitz function on $[0, 1]$, and let $\hat{f}$ denote its piecewise linear approximation
(\cref{def-piecewise-linear-approximation}) on the uniform grid $t\in\mathbb{R}^I$. Then
\[
\|\hat{f} - f\|_2^2 \leq \tfrac{2}{15}  \Delta t^2 \lipschitz{f}^2.
\]

\end{lemma}

\begin{proof}
See \cref{sec-piecewise-function-approximation-proof}.
\end{proof}

\cref{lem-piecewise-function-approximation} compares favourably with
existing results. De Boor~\cite[Ch.~III,~Eq.~17]{de_boor_practical_2001} establishes the bound
\[
\|\hat{f} - f\|_2^2 \leq \tfrac{1}{16} \Delta t^4 \|f''\|_2^2,
\]
which converges faster as the interval width
\(\Delta t \to 0\) (equivalently, \(I\to\infty\)), but requires
\(f\in C^2([0, 1])\) with bounded second
derivative. Lipschitz continuity and second-order differentiability
doesn't suffice: for example, the parametrized soft-plus function $f(x)=\ln(1+\exp(cx))/c$ is $1$-Lipschitz for all $c>0$, but has unbounded second derivative $f''(0)=c/2$ as $c\to \infty$.

De Boor showed the error vanishes for continuous
functions (\emph{not} necessarily Lipschitz), at a slower
rate than \cref{lem-piecewise-function-approximation}~\cite[Ch.~III,~Eq.~18]{de_boor_practical_2001}. Kunoth et al.~\cite[Sec.~1.5, Th.~18]{kunoth_splines_2018} generalizes the approximation error to functions in Sobolev spaces, achieving for \emph{differentiable} Lipschitz functions a rate comparable to
\cref{lem-piecewise-function-approximation}:
\begin{equation}\label{eq-kunoth-approximation-error}
  \|\hat{f} - f\|_2^2 \leq K \Delta t^2\|f'\|_2^2
\end{equation}
for some constant \(K>0\).
We should not expect a better rate because
Lipschitz functions are differentiable almost everywhere~\cite{heinonen_lectures_2004}.
For Lipschitz functions, the derivative, where it exists, satisfies \(\left\lvert f'(x) \right\rvert\leq \lipschitz{f}\), giving
\(\left\lVert f'\right\rVert_2^2\leq  \lipschitz{f}^2\). Thus,
\cref{lem-piecewise-function-approximation} explicitly calculates the constant
\(K=2/15\) in \cref{eq-kunoth-approximation-error}.

\begin{theorem}[Continuous Problem Connection]\protect\hypertarget{thm-problem-connection}{}\label{thm-problem-connection}
Let \(f:[0, 1]\to\mathbb{R}\) be \(\lipschitz{f}\)-Lipschitz with discretization
\(x^*\in\mathbb{R}^I\). Let \(x\in\mathbb{R}^I\), and construct the corresponding piecewise linear function \(\hat{f}_x\) (\cref{def-piecewise-linear-approximation}). For \(\epsilon > 0\), if
\(
I > C \epsilon^{-1} + 1
\)
and
\(
\|x -x^* \|_2^2 < D \epsilon
\)
then
\[
  \|\hat{f}_x - f\|_2 < \epsilon,
\]
where \(C = \sqrt{\frac{8}{15}} \lipschitz{f}\) and
\(D = \lipschitz{f}/\sqrt{30}\).
\end{theorem}

\begin{proof}
See \cref{sec-problem-connection-proof}.
\end{proof}

\Cref{thm-problem-connection} shows that an approximate solution \(x\) to the discrete problem \cref{eq-scale-s-discretized-problem} constructs a
piecewise linear function \(\hat{f}_x\) that approximately
solves the continuous problem \cref{eq-abstract-continuous-problem},
provided the grid is sufficiently fine. The constant \(C\) grows with the
Lipschitz constant \(\lipschitz{f}\), which
reflects the expected behaviour that functions with greater variation require finer grids.

\section{Cost comparison with a single-scaled approach}\label{comparison-with-projected-gradient-descent}

\Cref{cor-multiscale-convergence} established that two versions
of the multiscale algorithm (lazy and greedy) converge to a solution of
\cref{eq-scale-s-discretized-problem}. We now establish the paper's central claim: this
convergence can be obtained at lower total cost, and with a tighter expected error bound, than solving~\cref{eq-scale-s-discretized-problem} at the finest scale \(s=1\) along.

\paragraph{Cost model} Throughout this section we assume that one update $x^{k+1}\gets U(x^k)$  at problem size $I$ costs $\Theta(I^p)$ for some fixed $p\ge1$. The linear case $p=1$ covers the standard cost for projected gradient descent on separable problems: one gradient evaluation and one box projection each cost $\Theta(I)$ or $\Theta(I\log I)$. Denser problems cost more: the least-squares gradient update $x\leftarrow x-A^\top(Ax-y)$ with $A$, an $I$-by-$I$ matrix, scales quadratically ($p=2$). The multiscale advantage established below only grows with $p$, so $p=1$ is the worst case for our method.

\Cref{def-iterate-convergence} provides iterate convergence for a $q$-linear update \(U\). Using this, we derive expected convergence bounds for three approaches: applying many iterations of \(U\) only at scale $s=1$ (\cref{thm-expected-pgd-convergence}), greedy multiscale descent (\cref{thm-expected-greedy-convergence}), and lazy multiscale descent (\cref{thm-expected-lazy-convergence}). We demonstrate that both multiscale variants achieve tighter error bounds with lower computational cost when the problem size is sufficiently large.

\begin{theorem}[Expected Single-scale
Convergence]\protect\hypertarget{thm-expected-pgd-convergence}{}\label{thm-expected-pgd-convergence}
Consider problem \cref{eq-scale-s-discretized-problem} with $I=2^S+1$ discretization points for some $S\ge1$, solved at the finest resolution (scale \(s=1\)). Assume
solutions $x_{1}^*\in\mathbb{R}^{I}$ are either normalized
\(\left\lVert x_{1}^* \right\rVert _2 = 1\) or centered
\(\left\lVert x_{1}^* \right\rVert _2 = 0\). Given initialization
\(x_{1}^0 \in \mathbb{R}^I\) with i.i.d.\@
standard normal entries \(x^0_{1}[{i}]\sim \mathcal{N}(0,1)\), performing $K$ iterations of \(U\) yields the expected error
\begin{equation}\phantomsection\label{eq-pgd-expected-error}{
  \mathbb{E} \left\lVert x_{1}^{K}-x_{1}^* \right\rVert\leq q^{K} \sqrt{ 2^S + 2}.
}\end{equation}
\end{theorem}

\begin{corollary}[Expected Number of
Iterations]\protect\hypertarget{cor-expected-pgd-convergence}{}\label[corollary]{cor-expected-pgd-convergence}
If we iterate $U(x)$ \(K\) times, where \(
K \geq \left(\log(1/\epsilon)+{\log(2^S + 2)}/{2}\right)/(-\log q),
\)
then the expected error is bounded as
\(
\mathbb{E} \left\lVert x_{1}^{k} -x_{1}^* \right\rVert_2 \leq \epsilon.
\)

\end{corollary}

\begin{proof}
See \cref{sec-expected-pgd-convergence-proof}.
\end{proof}

\begin{theorem}[Expected Greedy Multiscale
Convergence]\protect\hypertarget{thm-expected-greedy-convergence}{}\label{thm-expected-greedy-convergence}
Assume the same setting as
\cref{thm-expected-pgd-convergence}, but use the greedy
multiscale descent method starting at scale \(s = S\).
Performing \(K_s\) iterations at each scale yields the expected final
error
\begin{equation}\phantomsection\label{eq-greedy-multiscale-expected-error}{
  \mathbb{E}\, \| x_{1}^{K_1} -x_{1}^* \|_2
  \leq
  q^{K_{1}}\sqrt{ 2 ^{S+1}}\left(q^{r_S}+\frac{\lipschitz{f^*}}{{2^{S+2}}}\sum\nolimits_{s=1}^{S-1}  2 ^{s} q^{r_s} \right)
}\end{equation}
where \(r_s = \sum_{t=2}^sK_{t}\) with the base case \(r_1 = 0\).
\end{theorem}

\begin{proof}
The result follows from \cref{thm-expected-pgd-convergence,thm-greedy-multiscale-descent-error} with the initial error bound \(\mathbb{E}\left[ \left\lVert x_{S}^{0} -x_{S}^* \right\rVert \right]\leq q^0 \sqrt{2^1 + 2} = 2\) starting at the coarsest scale with \(I=3\) points.
\end{proof}

Determining the iteration count required to achieve a specified accuracy is less direct for multiscale methods than for single-scaled ones. Moreover, a fair comparison requires accounting for computational cost: an iteration at coarse scale~$S$ requires fewer floating-point operations than an iteration at the finest scale $s=1$. We therefore configure each iteration of multiscale to simultaneously achieve lower total cost than fine-scale-only optimization while obtaining a tighter expected error bound. Specifically, we require \cref{eq-greedy-multiscale-expected-error} to be smaller than \cref{eq-pgd-expected-error}. This requires costing multiscale in terms of one iteration of \(U\) at the finest scale $s=1$.

\begin{lemma}[Cost of Greedy
Multiscale]\protect\hypertarget{lem-cost-of-gd-vs-ms}{}\label[lemma]{lem-cost-of-gd-vs-ms}
Suppose performing \(U\) with \(K\)
iterations at the finest scale \(s=1\) has total cost
\(C_{\mathrm{GD}}=CK\), where \(C>0\) is the cost per iteration at that scale.
Assume the cost \(C\) of one iteration of \(U\) scales
polynomially \(C = \Theta(I^p)\) with the problem size \(I\) for some constant \(p\geq 1\).
Then greedy multiscale descent with $K_s$ iterations at scale $s$ has total cost \(C_{\mathrm{GM}}\) that satisfies
\(
C_{\mathrm{GM}} \leq C\sum_{s=1}^S  K_{s} \left( \frac{3}{5} \right)^{s-1}.
\)

With \(K_s=1\) at every scale except the finest, where \(K_1 = K-2\), \(C_{\mathrm{GM}} < C_{\mathrm{GD}} \).
\end{lemma}

\begin{proof}
See \cref{sec-cost-of-gd-vs-ms-proof}.
\end{proof}

The analysis in \cref{lem-cost-of-gd-vs-ms} ignores the cost of
interpolation and memory allocation when changing iterate dimensions between scales. These costs become negligible as the
iteration counts increase. Empirical results in \cref{fig-loss-convergence-in-time,sec-numerical-experiments} confirm that greedy multiscale
remains cheaper in practice.

Having established conditions that ensure greedy multiscale is cheaper than the single-scaled approach, we now establish conditions that ensure tighter expected error bounds.

\begin{corollary}[Sufficient Conditions for Greedy Multiscale to be
Better]\protect\hypertarget{cor-greedy-cost}{}\label[corollary]{cor-greedy-cost}
Assume the base algorithm has iterate convergence (\cref{def-iterate-convergence}) with rate \(q\in (0,1/2)\).
Assume the finest scale problem has \(I = 2^{S} + 1\) points, where the number of scales \(S\) satisfies
\[
S \geq \max\left\{4,\, \log_2 \left( \frac{\lipschitz{f^*}}{\sqrt{2}q^2(1-2q)(1-\sqrt{2}q)}\right)\right\}.
\]
Then greedy multiscale with one iteration \(K_s=1\) at
each scale except the finest (where \(K_1 = K - 2\))
simultaneously achieves lower total cost and a tighter expected final error bound than performing \(K\) iterations of \(U\) at the finest scale.
\end{corollary}

\begin{proof}
See \cref{sec-greedy-cost-proof}.
\end{proof}

We establish analogous results for lazy multiscale.

\begin{lemma}[Cost of Lazy
Multiscale]\protect\hypertarget{lem-pgd-vs-lazy-multiscale}{}\label[lemma]{lem-pgd-vs-lazy-multiscale}
Suppose the same set up as \cref{lem-cost-of-gd-vs-ms}.
Lazy multiscale descent with $K_s$ iterations at scale $s$ has a total cost $C_{\mathrm{LM}}$ satisfying
\[
C_{\mathrm{LM}} \leq C \left(\textstyle\sum_{s=1}^{S-1} 2^{-s}K_{s} + 2^{-S}3K_{S}\right).
\]
With \(K_s = \left\lceil \frac{4}{5} K \right\rceil - 1\) iterations at each
scale,
\(
C_{\mathrm{LM}} < C_{\mathrm{GD}}.
\)

\end{lemma}

\begin{proof}
See \cref{sec-cost-of-gd-vs-lazy-ms-proof}
\end{proof}

\begin{theorem}[Expected Lazy Multiscale
Convergence]\protect\hypertarget{thm-expected-lazy-convergence}{}\label{thm-expected-lazy-convergence}
Assume the same setting as \cref{thm-expected-pgd-convergence},
but use the lazy multiscale starting at
coarsest scale \(S\) with \(I=2^S + 1\)
points at the finest scale.
Performing constant $K_s=K$ iterations at each scale yields the expected final error
\begin{equation}\label{eq-lazy-multiscale-error-bound}
\begin{aligned}
&\mathbb{E} \left\lVert x_{1}^{K} -x_{1}^* \right\rVert_2  \\
&\leq  d(K)\left( 2(d(K)+1)^{S-1} + \frac{\lipschitz{f^*}}{2} \frac{ \sqrt{2^S} \left( d(K) + 1 \right)^S - \sqrt{2} \left( d(K) + 1 \right)}{\sqrt{2^S} \left( d(K) + 1 \right) \left( \sqrt{2} d(K) + \sqrt{2} - 1 \right)} \right),
\end{aligned}
\end{equation}
where \(d(K)=q^K\) and the solution function \(f^*\) is
\(\lipschitz{f^*}\)-Lipschitz on \([0, 1]\).
\end{theorem}

\begin{proof}
The proof is similar to \cref{thm-expected-greedy-convergence} by using \cref{thm-lazy-multiscale-descent-error}.
\end{proof}

We establish sufficient conditions that ensure lazy multiscale achieves both lower cost and tighter error bounds than the single-scaled approach.

\begin{corollary}[Sufficient Conditions for Lazy Multiscale to be
Better]\protect\hypertarget{cor-lazy-cost}{}\label[corollary]{cor-lazy-cost}
Assume \(U\) has iterate convergence (\cref{def-iterate-convergence}) with rate
\(q\in (0,1)\).
Assume the single-scaled approach runs for
\(K>5(\log_q(\sqrt{2}-1)-1)/4\) iterations, and lazy multiscale runs for
\(K'=\lceil 4K/5 \rceil -1\) iterations at every scale. If the finest scale has
at least \(I=2^S+1\) points where
\[
S\geq {\log\left( \frac{q^{-K/5-1}}{1+q^{K'}}\left( 2+\frac{\lipschitz{f^*}}{2(\sqrt{ 2 }(1+q^{K'}) -1)} \right) \right)} \Big/ {\log\left( \frac{\sqrt{ 2 }}{1+q^{K'}} \right)},
\]
then lazy multiscale is cheaper and achieves a tighter expected error bounds than the single-scaled approach.
\end{corollary}

\begin{proof}
See \cref{sec-lazy-cost-proof}.
\end{proof}

\Cref{cor-greedy-cost,cor-lazy-cost} reveal several insights. Both
multiscale variants require a minimum number of scales \(S\) (equivalently a minimum sized problem \(I\)) for the multiscale optimization to outperform single scale methods.
\Cref{thm-problem-connection} establishes that arbitrarily small error between continuous and discretized solutions requires arbitrarily large problem size $I$. Both corollaries show that the required $I$ increases with the Lipschitz constant $\lipschitz{f^*}$, consistent with \cref{thm-problem-connection}. A subtle requirement
is that the base algorithm cannot have iterative convergence with $q=0$, which would enable solving the discretized problem in one iteration, and thus eliminate the benefit of multiscale.

\Cref{cor-greedy-cost,cor-lazy-cost} differ in their constraints on
the convergence rate $q$. Greedy multiscale with one iteration per scale requires convergence rate $q\in (0,1/2)$ because interpolation error must remain smaller than the algorithmic progress at each scale. Lazy multiscale, on the other hand, allows $q\in (0,1)$ but requires a minimum of $K'>\log_q(\sqrt{2}-1)$ iterations at each scale to sufficiently control error accumulation. This is a mild requirement: two iterations per scale suffice for lazy multiscale to improve on the base algorithm when $0<q<(\sqrt{2} -1)^{1/2}\approx 0.6436$, and both algorithms require arbitrarily many iterations to drive the final error to zero.

\section{Numerical experiments and
benchmarks}\label{sec-numerical-experiments}

The experiments serve two purposes. First, they confirm the cost predictions of \cref{comparison-with-projected-gradient-descent} on the controlled one-dimensional motivating example of \cref{sec-motivating-example}, revisited in \cref{sec-motivating-example-numerics}. Second, they show that the one-dimensional setting of the theory is not a practical restriction. \Cref{the-optimization-problem} formulates density demixing---recovering mixtures of continuous probability densities from noisy samples---as a constrained Tucker-1 tensor factorization, and \cref{sec-synthetic-probability-unmixing,sec-geological-probability-unmixing} apply the greedy multiscale method to synthetic and real geological instances. The synthetic experiments also test the one-iteration-per-coarse-scale prescription of \cref{cor-greedy-cost} by varying the number of coarse-scale iterations.

The projected-gradient multiscale and single-scale algorithms are implemented in the Julia package \texttt{BlockTensorFactorization.jl}~\cite{Richardson_BlockTensorFactorization_jl}: the function \texttt{factorize} implements projected gradient descent for the Tucker-1 decomposition problem at a single scale, and \texttt{multiscale\_factorize} implements the greedy multiscale method (\cref{alg-greedy-multiscale}).
All experiments ran on an Intel Core i7-1185G7 with 32GB of RAM, without parallelization.

\subsection{Motivating example: numerics} \label{sec-motivating-example-numerics}

We first compare an instance of the greedy multiscale method
(\cref{alg-greedy-multiscale}) against single-scale optimization on the
motivating example from \cref{sec-motivating-example}. Full experimental
details are provided in \cref{sec-motivating-example-details}.

\begin{figure}[t]
  \centering
  \includegraphics[width=0.49\linewidth]{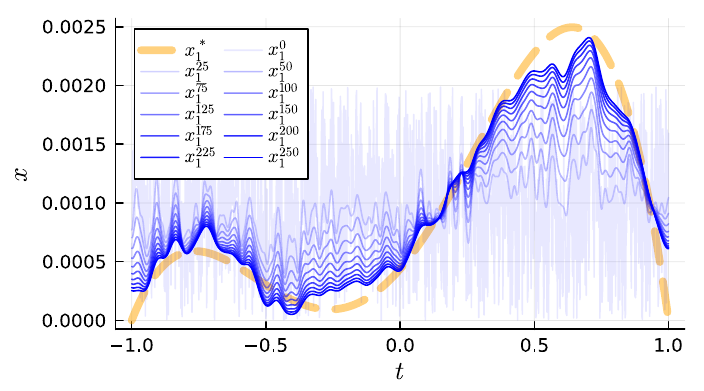}
  \includegraphics[width=0.49\linewidth]{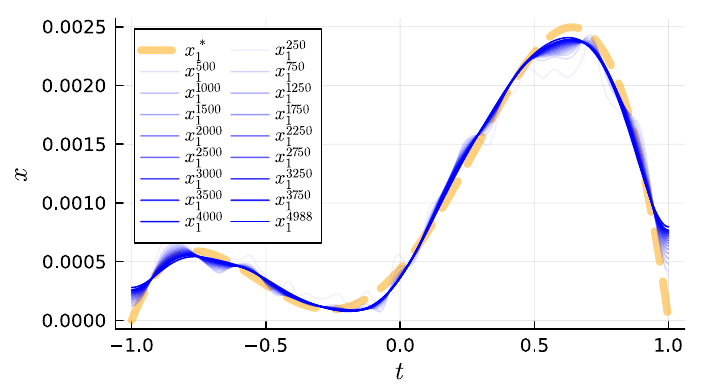}
  \includegraphics[width=0.49\linewidth]{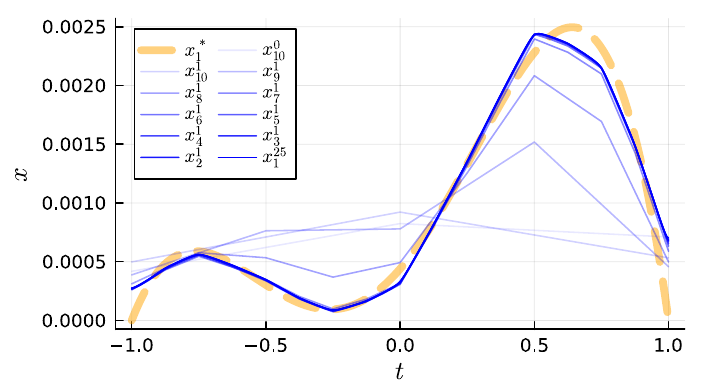}
  \includegraphics[width=0.49\linewidth]{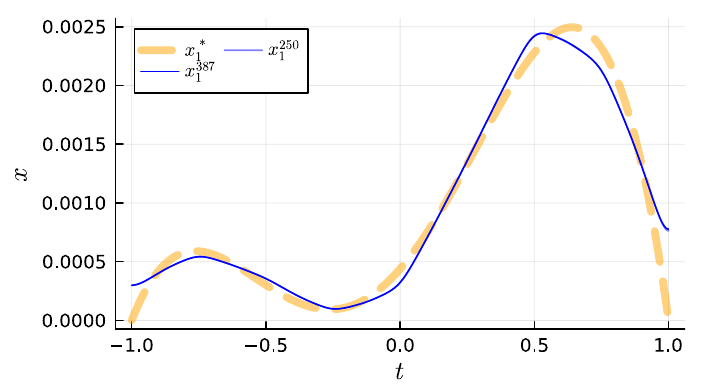}
  \caption{Typical progression of iterates \(x_s^k\) for scale \(s\) and iteration \(k\) using single scale (top row) and multiscale (bottom row) approaches for the motivating example from \cref{sec-motivating-example}. The clean ground truth is plotted behind in the thick dashed lined. }\label{fig-iterate-progression}
\end{figure}

\Cref{fig-multi-vs-single-time} shows the improvement in computation time versus problem size achieved by the multiscale approach.
\Cref{fig-iterate-progression} shows the typical convergence of iterates \(x_s^k\) for single scale and multiscale approaches.
And \cref{fig-loss-convergence-in-time} shows the typical loss convergence for the single scale and multiscale methods.

We call multiscale's tendency to smooth iterates an implicit regularization. This helps smooth solutions by carrying information from farther away points to speed up convergence. In this problem, the graph laplacian regularizer only looks at the immediate neighbouring points. So working at coarser scales lets us compare points that are not immediately neighbouring at the finest scale.

\begin{figure}[t]
  \centering
  \includegraphics[width=0.49\linewidth]{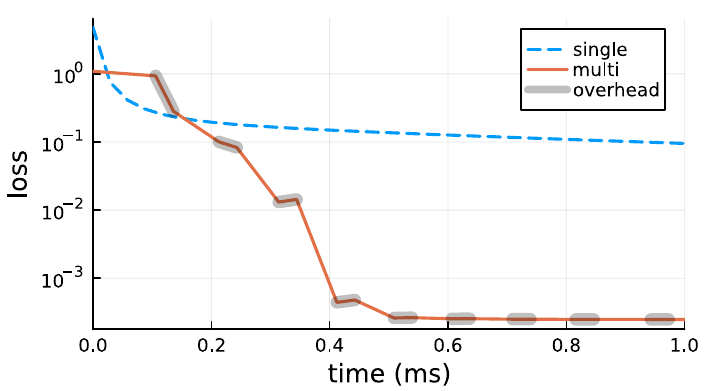}
  \includegraphics[width=0.49\linewidth]{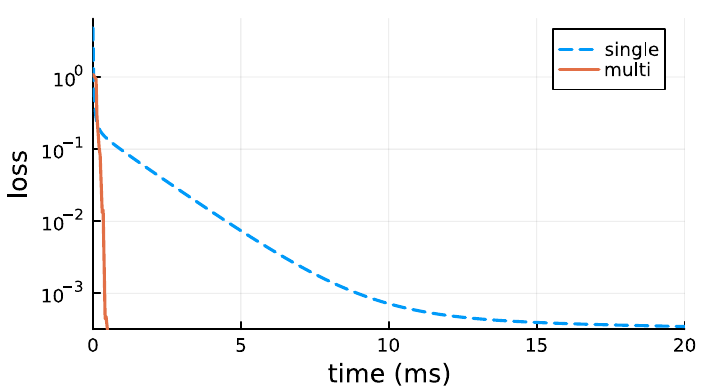}
  \caption{Typical loss convergence using single scale (dashed blue) and multiscale (solid orange) approaches for the motivating example from \cref{sec-motivating-example}. The left plot shows the first millisecond, and the right plot shows the first 20 milliseconds. The left plot also highlights in grey regions that are not from projected gradient descent steps. These are the result of minimal overhead with the multiscale method such as interpolation, allocations, and extra function calls.}\label{fig-loss-convergence-in-time}
\end{figure}

\subsection{Density demixing as optimization}\label{the-optimization-problem}

Graham et.\ al.~\cite{graham_tracing_2025} provides a full treatment of this formulation which we summarize here. Additional details are given in \cref{additional-tensor-details-and-reference}.

Given \(I\) mixtures \(y_i:\mathcal{D}\subseteq\mathbb{R}^N\to\mathbb{R}_+\) of \(R\) source probability density functions
\(b_r:\mathcal{D}\subseteq\mathbb{R}^N\to\mathbb{R}_+\),
\(
y_i(x) = \textstyle\sum_{r\in[R]} a_{i,r} b_r(x),
\)
we seek to recover the sources \(\{b_r\}\) and their mixing coefficients
\(a_{i,r}\). We assume \(R<I\) and that the source densities \(\{b_r\}\) are
linearly independent, which ensures this is a well-posed problem.

The continuous problem formulation is
\begin{subequations}\label{eq-continuous-decomposition}
\begin{equation*}\label{eq-continuous-decomposition-problem}
  \min_{\{a_{i,r}\}, \{b_r\}}
  \half\textstyle\sum_{i\in[I]} \left\lVert \sum_{r\in[R]} a_{i,r} b_r  - y_i\right\rVert_2^2
\end{equation*}
such that for all \(i\in[I]\),
\begin{alignat}{2}
    1 &= \textstyle\sum_{r\in[R]} a_{i,r}, \qquad & a_{i,r} &\geq 0, \label{eq-continuous-decomposition-constraints-a}\\
    1 &= \textstyle\int_{\mathcal{D}} b_r(x)\, dx, \qquad & b_r(x) &\geq 0 \quad \forall x\in\mathcal{D}, \ \forall r\in[R]. \label{eq-continuous-decomposition-constraints-b}
\end{alignat}
\end{subequations}

Discretizing yields the Tucker-1 decomposition problem (see \cref{def-tucker-1-decomposition})
\[
\min\Set{
  \tfrac{1}{2}\left\lVert B \times_1 A  - Y \right\rVert_F^2
  | A\in\Delta_{R}^I,\ B\in\Delta_{K_1\dots K_N}^R
},
\]
\begin{subequations}\label{eq-constrained-least-squares}
\begin{align}
\Delta_{R}^I &= \Set{A\in\mathbb{R}_+^{I\times R} | \textstyle\sum_{r=1}^R A[i,r] = 1, \ \forall i\in [I]}, \\
\Delta_{K_1\dots K_N}^R &= \Set{B\in\mathbb{R}_+^{R\times K_1\times\dots \times K_N} | \textstyle\sum_{k_1,\dots,k_N=1}^{K_1\dots K_N} B[r,k_1,\dots,k_N] = 1, \ \forall r\in [R]}.
\end{align}
\end{subequations}

In the geological data example (\cref{sec-geological-probability-unmixing}), each dimension of \(b_r\) is independent. This allows a simpler discretization, so that \(B\) becomes a third-order tensor with
\[
B\in\Delta_{K}^{RJ}=\Set{B\in\mathbb{R}_+^{R\times J \times K} | \textstyle\sum_{k=1}^{K} B[r,j,k] = 1, \ \forall (r,j)\in [R]\times[J]},
\]
and we reinterpret \(Y\in\mathbb{R}_+^{I\times J\times K}\) accordingly.

\subsection{Synthetic data}\label{sec-synthetic-probability-unmixing}

We generate three source distributions for a synthetic test. Each
distribution is a 3-dimensional product of standard distributions; see
\cref{sec-synthetic-data-setup} and this paper's GitHub
repository~\cite{Richardson_multiscale_paper_code} for details.
\Cref{tab-synthetic-benchmark} reports benchmarks of the two approaches
across 20 samples.

\begin{table}[t]
\caption{Benchmarks over 20 runs of the synthetic density demixing problem:
single-scale \texttt{factorize} versus \texttt{multiscale\_factorize}.
Columns report the median and mean ($\pm$ standard deviation) wall-clock
time, the range of times, the median share of time spent in garbage
collection (GC), and the total memory allocated.}
\label{tab-synthetic-benchmark}
\centering
\small
\begin{tabular}{lccccc}
\hline
Method & median & mean $\pm\ \sigma$ & min--max & GC & memory\\
\hline
single-scale & 2.412\,s & 2.418\,s $\pm$ 0.628\,s & 1.305--3.487\,s & 27.7\,\% & 1.46\,GiB\\
multiscale & 0.335\,s & 0.456\,s $\pm$ 0.313\,s & 0.231--1.583\,s & 11.6\,\% & 0.36\,GiB\\
\hline
\end{tabular}
\end{table}

The function \texttt{multiscale\_factorize} runs roughly seven times faster and uses about one-fourth of the memory. This comparison includes overhead such as interpolating the tensor and repeatedly calling the internal
\texttt{factorize} function, which represents approximately \(4\%\) of the total time. The specific numbers are less important than the qualitative comparison: both algorithms could be optimized and benchmarked on state-of-the-art hardware, but this basic implementation demonstrates that multiscale methods can accelerate single-scaled algorithms.

In these \texttt{multiscale\_factorize} benchmarks, rather than performing one iteration at each scale coarser than the finest scale, we advance to the next scale once sufficient progress has been made; see \cref{sec-synthetic-data-setup} for details.
This differs from our analysis in \cref{cor-greedy-cost}, which suggests using exactly \(K_s=1\) at coarse scales. However, running additional iterations at coarser scales often reduces the total number of iterations required at the finest scale.

\Cref{fig-time-vs-coarse-iterations} (left) shows median runtime as a function of the number of fixed coarse iterations \(K_s\) in multiscale for \(s=S,S-1,\dots,2\), where we run iterations at the finest scale until the objective falls below \(10^{-6}\). This figure demonstrates that performing more than one iteration at coarser scales improves multiscale performance, but only to a point: in \cref{fig-time-vs-coarse-iterations} (left), more than 12 iterations at each coarse scale wastes time. \Cref{fig-time-vs-coarse-iterations} (right) illustrates this by showing the number of fine-scale iterations \(K_1\) needed to converge to a minimum value. Our implementation of \texttt{multiscale\_factorize} \cite{Richardson_BlockTensorFactorization_jl} addresses this by allowing multiple iterations at coarser scales and advancing to the next scale when other criteria are met, such as small gradients or objective values.

Both plots in \cref{fig-time-vs-coarse-iterations} show that multiscale with exactly \(K_s=1\) iterations at coarse scales \(s=S,S-1,\dots,2\) performs slightly worse than single-scale optimization. We attribute this to interpolation and to overhead costs in the multiscale approach. The problem size of \(2^6+1=65\) points along each continuous dimension may be too small to see immediate benefit from multiscale when only one iteration is performed at each scale; larger problems may needed to realize the advantage illustrated by \cref{fig-multi-vs-single-time}. However, flexible criteria to advance between scales may make multiscale competitive even for smaller problems as shown in the benchmark tests.

\begin{figure}[t]
  \centering
  \includegraphics[width=0.49\linewidth]{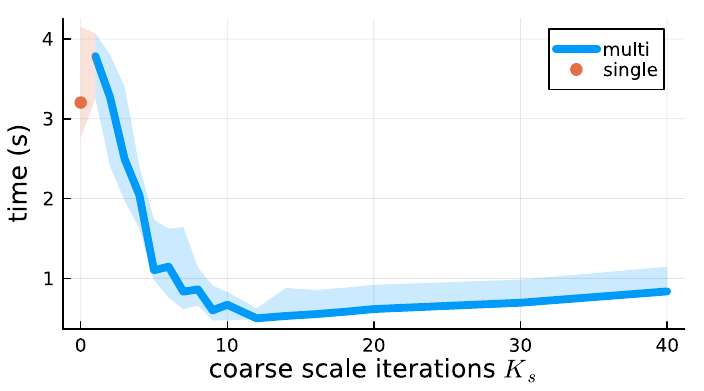}
  \includegraphics[width=0.49\linewidth]{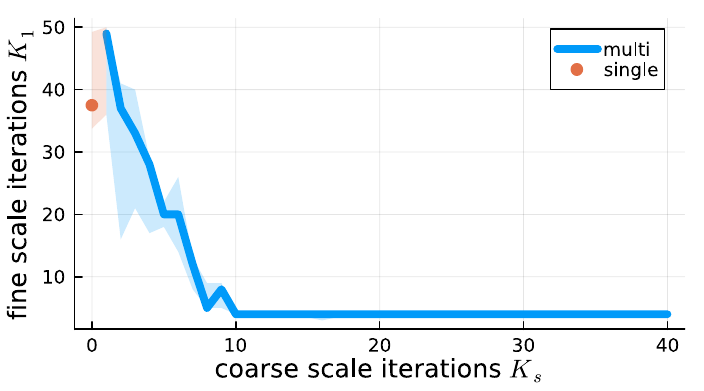}
  \caption{(Left) Median time until convergence at the finest scale \(s=1\) as a function of number of fixed iterations \(K_s\) at coarser scales \(s=S,S-1,\dots,2\). (Right) Median number of iterations \(K_1\) at the finest scale \(s=1\). Shaded ribbon shows the bottom and top quartile times (left) and iterations (right).}\label{fig-time-vs-coarse-iterations}
\end{figure}

\subsection{Real geological
data}\label{sec-geological-probability-unmixing}

We use the same sedimentary data and Tucker-$1$ model described in Graham
et.\ al.~\cite{graham_tracing_2025} to factorize a tensor containing
mixtures of estimated densities. \Cref{sec-geo-numerics-details} provides
the details, and the full code appears in
\newline\texttt{distribution\_unmixing\_geology.jl} in this paper's GitHub
repository~\cite{Richardson_multiscale_paper_code}.
\Cref{tab-geological-benchmark} reports the benchmarks for
\texttt{factorize} and \texttt{multiscale\_factorize}.

\begin{table}[t]
\caption{Benchmarks of the geological density demixing problem:
single-scale \texttt{factorize} versus \texttt{multiscale\_factorize}.
Columns as in \cref{tab-synthetic-benchmark}.}
\label{tab-geological-benchmark}
\centering
\small
\begin{tabular}{lccccc}
\hline
Method & median & mean $\pm\ \sigma$ & min--max & GC & memory\\
\hline
single-scale & 417\,ms & 389\,ms $\pm$ 109\,ms & 174--594\,ms & 19.9\,\% & 361\,MiB\\
multiscale & 91\,ms & 96\,ms $\pm$ 19\,ms & 78--154\,ms & 0.0\,\% & 106\,MiB\\
\hline
\end{tabular}
\end{table}

By every metric, the multiscale approach is faster and uses less memory than the standard single-scale factorization. Comparing median time and memory estimates, the multiscale method runs roughly four times faster with about a third of the memory.

\section{Conclusion}\label{conclusion}

We developed a multiscale approach for optimizing Lipschitz continuous functions by solving discretized problems at progressively finer scales. Our convergence analysis establishes that both greedy and lazy variants achieve the finest-scale solution through per-scale descent combined with controlled interpolation error. The framework extends to linear and norm constraints and tensor-valued functions.

The explicit error bounds derived quantify the tradeoff between discretization granularity and solution accuracy, and prove that multiscale outperforms direct fine-scale optimization when the problem size and Lipschitz constant are sufficiently large.

Several directions merit further investigation. Our analysis assumes a fixed number of base algorithm iterations at each scale. However, an adaptive iteration count that responds to convergence progress could improve efficiency. With regard to the discretization, the dyadic scheme might be replaced by adaptive grid refinement based on local Lipschitz estimates. Second-order base algorithms present another avenue, though their cost would be better characterized in terms of linear system solves rather than gradient evaluations. Finally, interpreting the pointwise discretization as a discrete wavelet transform suggests generalizations to other wavelet expansions.

\bibliographystyle{jabbrv_siam}
\bibliography{references}

\appendix

\section{Proofs}

In the following proofs, we take the unlabeled norm
\(\left\lVert \cdot \right\rVert = \left\lVert \cdot \right\rVert_2\) to
be the \(2\)-norm.

\newcommand{\prooftitle}[2]{Proof of \cref{#1} (#2)}
\pdfstringdefDisableCommands{%
  \def\prooftitle#1#2{Proof of #2}%
}

\subsection{\prooftitle{prp-linear-constraint-scaling}{Single linear constraint
scaling}}\label{sec-linear-constraint-scaling-proof}

For an even number of points \(I\), we have
\begin{align*}
&\left\lvert 2\left\langle \coarse{a}, \coarse{x} \right\rangle - b \right\rvert = \left\lvert 2\textstyle\sum_{i\text{ odd}} a[i]x[i] -\textstyle\sum_{i=1}^I a[i]x[i]\right\rvert  = \Big\lvert  \textstyle\sum_{i\text{ odd}}a[i]x[i] - \textstyle\sum_{i\text{ even}} a[i]x[i]\Big\rvert \\
&\leq \textstyle\sum_{i\text{ odd}}\left\lvert a[i]x[i]-a[i+1]x[i+1]
 \right\rvert
\leq \textstyle\sum_{i\text{ odd}}\lipschitz{fg}\Delta t\left\lvert i - (i+1) \right\rvert = \tfrac{I}{2}\lipschitz{fg} \tfrac{1}{I-1}.
\end{align*}
The second inequality follows from
\cref{cor-lipschitz-vectors-and-functions} in combination with
the product \(a[i] x[i] = f(t[i])g(t[i])\)
being Lipschitz with constant
\(
\lipschitz{fg} = (\lipschitz{f} \left\lVert g \right\rVert_{\infty} + \lipschitz{g} \left\lVert f \right\rVert_{\infty})
\)
; see \cref{lem-product-of-lipschitz-functions}.
Dividing by \(2\) proves the first statement.

For odd \(I\), we have
\begin{align*}
&\left\lvert 2\left\langle \coarse{a}, \coarse{x} \right\rangle - (1 + \tfrac{1}{I})b \right\rvert =
\left\lvert \textstyle\sum_{i\text{ odd}} a[i]x[i] - \textstyle\sum_{i\text{ even}} a[i]x[i] -\tfrac{1}{I}b\right\rvert \\
&= \left\lvert \textstyle\sum_{j=1}^{(I-1)/2} a[2j-1]x[2j-1] + a[i]x[i] - \textstyle\sum_{j=1}^{(I-1)/2} a[2j]x[2j] -\tfrac{1}{I}b\right\rvert \\
&\leq \textstyle\sum_{j=1}^{(I-1)/2} \left\lvert  a[2j-1]x[2j-1] - a[2j]x[2j] \right\rvert + \tfrac{1}{I} \left\lvert I a[i]x[i] -b\right\rvert \\
&= \textstyle\sum_{j=1}^{(I-1)/2} \left\lvert  a[2j-1]x[2j-1] - a[2j]x[2j] \right\rvert + \tfrac{1}{I} \textstyle\sum_{i=1}^{I}\left\lvert a[i]x[i] - a[i]x[i]\right\rvert
\end{align*}
The first sum is bounded by \(\tfrac{I-1}{2}\lipschitz{fg} \frac{1}{I-1} = \lipschitz{fg}\tfrac{1}{2}\)
in a similar manner to the even \(I\) case. The second term is bounded by
\(
\tfrac{1}{I} \textstyle\sum_{i=1}^{I}\lipschitz{fg}\Delta t \left\lvert I - i \right\rvert
= \tfrac{1}{I} \lipschitz{fg}\tfrac{1}{I-1} \tfrac{I(I-1)}{2}= \lipschitz{fg}\tfrac{1}{2}.
\)
Adding these terms together gives us the bound \(\left\lvert 2\left\langle \coarse{a}, \coarse{x} \right\rangle - \tfrac{I+1}{I}b \right\rvert \leq \lipschitz{fg}\).
Dividing by \(2\) completes the proof.

\begin{lemma}[Product of Lipschitz
Functions]\protect\hypertarget{lem-product-of-lipschitz-functions}{}\label[lemma]{lem-product-of-lipschitz-functions}
Let
\(f,g:\mathcal{D}\subseteq\mathbb{R}^{I_1\times \dots \times I_N}\to\mathbb{R}\)
be \(\lipschitz{f}\) and \(\lipschitz{g}\) Lipschitz respectively. Also
assume \(f\) and \(g\) are bounded with
\(f(A)\leq \left\lVert f \right\rVert_\infty\) and
\(g(A)\leq \left\lVert g \right\rVert_\infty\) for all
\(a\in\mathcal{D}\). Then the product function \((fg)(A)=f(A)g(A)\) is
Lipschitz with constant
\[
\lipschitz{fg} =  \lipschitz{f} \left\lVert g \right\rVert_\infty +\lipschitz{g} \left\lVert f \right\rVert_\infty .
\]

\end{lemma}

\begin{proof}
Let \(f,g:\mathcal{D}\to \mathbb{R}^{I_1\times \dots \times I_N}\), and
\(A,B\in\mathcal{D}\). \[
\begin{aligned}
\left\lvert f(A)g(A)-f(B)g(B) \right\rvert&= \left\lvert f(A)g(A)-f(A)g(B)+f(A)g(B)- f(B)g(B)\right\rvert\\
&\leq \left\lvert f(A) \right\rvert\left\lvert g(A)-g(B) \right\rvert + \left\lvert g(B) \right\rvert\left\lvert f(A)-f(B) \right\rvert \\
&\leq \left\lvert f(A) \right\rvert \lipschitz{g}\left\lVert A-B \right\rVert_F+ \left\lvert g(B) \right\rvert \lipschitz{f}\left\lVert A-B \right\rVert_F \\
&\leq \left\lVert f \right\rVert _{\infty}\lipschitz{g}\left\lVert A-B \right\rVert_F+\left\lVert g \right\rVert _{\infty}\lipschitz{f}\left\lVert A-B \right\rVert_F \\
&= \left(\left\lVert f \right\rVert _{\infty}\lipschitz{g}+\left\lVert g \right\rVert _{\infty}\lipschitz{f}\right)\left\lVert A-B \right\rVert_F
\end{aligned}
\]
\end{proof}

\subsection{\prooftitle{lem-lipschitz-interpolation}{Lipschitz function interpolation}}\label{sec-lipschitz-interpolation-proof}

Let \(f:\mathbb{R}^n\to\mathbb{R}\) be \(\lipschitz{f}\)-Lipschitz,
\(a,b\in\mathbb{R}^n\), and \(\lambda = \lambda_1\), \(1-\lambda = \lambda_2\) for \(\lambda\in[0,1]\).
Our goal is to bound
\(
\left\lvert ((1-\lambda) f(a)+\lambda f(b)) - f((1-\lambda) a+\lambda b) \right\rvert.
\)
We achieve the tightest bound
by separating
\(
f((1-\lambda) a+\lambda b)=  (1-\lambda) f((1-\lambda) a+\lambda b) + \lambda f((1-\lambda) a+\lambda b)
\)
and using triangle inequality to get
\[
\begin{aligned}
&\left\lvert ((1-\lambda) f(a)+\lambda f(b)) - f((1-\lambda) a+\lambda b) \right\rvert\\
&\leq \left\lvert (1-\lambda) f(a)-(1-\lambda) f((1-\lambda) a+\lambda b)\right\rvert  + \left\lvert \lambda f(b) - \lambda f((1-\lambda) a+\lambda b) \right\rvert \\
&\leq (1-\lambda) \lipschitz{f}\left\lVert a-((1-\lambda) a+\lambda b)\right\rVert _2  +\lambda  \lipschitz{f}\left\lVert b - ((1-\lambda) a+\lambda b) \right\rVert _2 \\
&= 2\lambda(1-\lambda) \lipschitz{f}\left\lVert a-b\right\rVert _2.
\end{aligned}
\]

\subsection{\prooftitle{lem-exact-interpolation}{Exact interpolation}}\label{sec-exact-interpolation-proof}

Using \cref{lem-lipschitz-interpolation} with \(\lambda=1/2\), we
have for even \(j\), \[
\begin{aligned}
\left\lvert x_s[{j}]-x_s^*[j] \right\rvert &= \left\lvert  \tfrac{1}{2}\hspace{-0.5mm}\left( x_{s+1}\left[{\tfrac{j}{2}}\right] + x_{s+1}\left[{\tfrac{j}{2}+1}\right] \right) - f(t_{s}\left[j\right])\right\rvert \\
&= \left\lvert  \tfrac{1}{2}\hspace{-0.5mm}\left( f\left(  t_{s+1}\left[{\tfrac{j}{2}}\right] \right) + f\left(  t_{s+1}\left[{\tfrac{j}{2}}+1\right] \right) \right) - f\left( \tfrac{1}{2}\hspace{-0.5mm}\left( {t_{s+1}\left[{\tfrac{j}{2}}\right]+t_{s+1}\left[{\tfrac{j}{2}+1}\right]}\right) \right)\right\rvert \\
 &\leq \tfrac{\lipschitz{f}}{2}\hspace{-0.5mm}\left\lvert t_{s+1}\left[{\tfrac{j}{2}}\right]-t_{s+1}\left[{\tfrac{j}{2}+1}\right]\right\rvert =\tfrac{\lipschitz{f}}{2}\tfrac{1}{I-1}. %
\end{aligned}
\]
For odd \(j\), the interpolated values match the function exactly
\[
x_s[j]=\interp{x}_{s+1}[j]=x_{s+1}\left[\tfrac{j+1}{2}\right]=x_{s+1}^*\left[\tfrac{j+1}{2}\right]=f\left(t_{s+1}\left[\tfrac{j+1}{2}\right]\right)=f(t_{s}[j])=x_{s}^*[j].
\]
Summing over all \(j\), we bound the squared error \[
\begin{aligned}
\left\lVert x_s-x_s^* \right\rVert_{2}^2&= \textstyle\sum_{j=1}^J\left\lvert x_s[j]-x_s^*[j] \right\rvert ^2=\textstyle\sum_{j \text{ even}}\left\lvert x_s[j]-x_s^*[j] \right\rvert^2\\
&\leq \textstyle\sum_{j\text{ even}} \tfrac{\lipschitz{f}^2}{4(I-1)^2} \leq (I-1) \tfrac{\lipschitz{f}^2}{4(I-1)^2} =\tfrac{\lipschitz{f}^2}{4(I-1)}.
\end{aligned}
\]
Taking square roots completes the proof.

\subsection{\prooftitle{lem-inexact-interpolation}{Inexact interpolation}}\label{sec-inexact-interpolation-proof}

Let the inexact values be \(x_{s+1}[i]=x_{s+1}^*[i]+\delta [i]\), and
interpolate to get \(x_{s}=\interp{x}_{s+1}\):
\[
\begin{aligned}
  &\interp{x}_{s+1}[j]\\
  &=\begin{cases}
  x_{s+1}[\frac{j+1}{2}]\hspace{23.7mm}=x_{s+1}^*[\frac{j+1}{2}] +\delta [{\frac{{j+1}}{2}}],  &\hspace{-2mm} \text{$j$ odd},\\
  \tfrac{1}{2}\hspace{-0.5mm}\left( x_{s+1}[\frac{j}{2}] + x_{s+1}[\frac{j}{2}+1] \right) =\tfrac{1}{2}\hspace{-0.5mm}\left( x_{s+1}^*[\frac{j}{2}] + x_{s+1}^*[\frac{j}{2}+1] +\delta [{\frac{j}{2}}]+\delta [{{\frac{j}{2}}+1}] \right), &\hspace{-2mm} \text{$j$ even.}
  \end{cases}
\end{aligned}
\]
Let \(e = \interp{x}_{s+1}[j]-\interp{x}_{s+1}^*[j] \).
Note \(\interp{x}_{s+1}^*\) is the interpolated vector of \(x_{s+1}^*\)
where \(x_{s+1}^*[i]=f(t_{s+1}[i])\), and is possibly different from the
exact discretization \(x_{s}^*[j]=f(t_{s}[j])\). Of course the entries
\(\interp{x}_{s+1}^*[j]=x_{s}^*[j]\) for odd \(j\).
We can now bound \(\lVert e \rVert\) in terms of \(\delta\).
\begin{align*}
\lVert e \rVert^2 &= \sum_{j\text{ odd}}\hspace{-1mm}  \left\lvert\delta \left[{\tfrac{{j+1}}{2}}\right] \right\rvert^2 + \sum_{j\text{ even}}\hspace{-1.5mm}  \left\lvert \tfrac{1}{2} \left( \delta \left[{\tfrac{j}{2}}\right] + \delta \left[{{\tfrac{j}{2}}+1}\right] \right)\right\rvert^2  = \left\lVert \delta \right\rVert^2 + \tfrac{1}{4}\sum_{i=1}^{I-1} \left( \delta [i]+\delta [{i+1}] \right)^2 \\
&\leq \left\lVert \delta \right\rVert^2 + \tfrac{1}{4}\left(\sum_{i=1}^{I} \delta [i]^2 +  \sum_{i=0}^{I-1}\delta [{i+1}]^2 +2\sum_{i=1}^{I-1}\delta [i]\delta [{i+1}] \right) \\
&\leq \tfrac{3}{2}\left\lVert \delta \right\rVert^2 +\tfrac{1}{2}\left(  \sum_{i=1}^{I}\delta [i]^2\right)^{1/2}\left(\sum_{i=0}^{I-1}\delta [{i+1}]^2\right)^{1/2}   \text{(Cauchy-Schwarz, add \(\delta [I]\), \(\delta [0]\))}\\
&= \tfrac{3}{2}\left\lVert \delta \right\rVert^2 +\tfrac{1}{2}\left\lVert \delta \right\rVert\left\lVert \delta \right\rVert = 2\left\lVert \delta \right\rVert^2
\end{align*}
Taking square roots and substituting gives us \(
\left\lVert \interp{x}_{s+1}-\interp{x}_{s+1}^* \right\rVert \leq \sqrt{ 2 } \left\lVert x_{s+1}- x_{s+1}^* \right\rVert.
\)

We use triangle inequality and \cref{lem-exact-interpolation} to bound the difference between
the interpolated approximate values \(x_s=\interp{x}_{s+1}\) and the
true values \(x_{s}^*\): \[
\begin{aligned}
\left\lVert x_s - x_{s}^* \right\rVert _{2} &\leq \left\lVert x_s - \interp{x}_{s+1}^* \right\rVert _{2} + \left\lVert \interp{x}_{s+1}^*-x_{s}^* \right\rVert _{2} = \left\lVert \interp{x}_{s+1} - \interp{x}_{s+1}^* \right\rVert _{2}+ \left\lVert \interp{x}_{s+1}^*-x_{s}^* \right\rVert _{2}\\
&\leq \sqrt{ 2 } \left\lVert x_{s+1}-x_{s+1}^* \right\rVert _{2} + \tfrac{\lipschitz{f}}{2\sqrt{ I-1 }} .
\end{aligned}
\]

In the following proofs, we take the unlabeled norm
\(\left\lVert \cdot \right\rVert = \left\lVert \cdot \right\rVert_2\) to
be the \(2\)-norm.

\subsection{\prooftitle{thm-greedy-multiscale-descent-error}{Greedy multiscale error bound}}\label{sec-greedy-multiscale-descent-error-proof}

Using \cref{def-iterate-convergence},
\(\left\lVert {x}^{K_{s}}_{s}-x^*_{s} \right\rVert  = q^{K_{s}} \left\lVert x_{s}^{0} -x_{s}^* \right\rVert\) so by \cref{lem-inexact-interpolation} we have
\[
\left\lVert x^0_{s-1} - x^*_{s-1} \right\rVert = \left\lVert \interp{x}^{K_{s}}_{s} - x^*_{s-1} \right\rVert
\leq\sqrt{ 2 } q^{K_{s}} \left\lVert x_{s}^{0} -x_{s}^* \right\rVert + \tfrac{L}{2\sqrt{2^{S-s+1}}}.
\]
Substituting \(s \mapsto s+1\) and letting \(e_{s}=\left\lVert x^0_{s} - x^*_{s} \right\rVert\) and
\(C=\frac{\lipschitz{f^*}}{2\sqrt{2^{S+1}}}\)
gives us
\[
\begin{aligned}
e_{s}
&\leq \sqrt{ 2 } q^{K_{s+1}} e_{s+1} + \sqrt{ 2^{s+1} } C.
\end{aligned}
\]
Applying \(U\) on the finest scale and using this inequality recursively for \(s=1,\dots,S-1\) gives us the desired result, where \(r_s = \sum_{t=1}^s K_t\),
\[
\lVert x_{1}^{K_{1}}-x_{1}^* \rVert \leq q^{K-1} e_1 \leq \sqrt{ 2 ^{S-1}}q^{r_S}e_{S} + C\textstyle\sum_{s=1}^{S-1}  2 ^{s} q^{r_s}.
\]

\subsection{\prooftitle{thm-lazy-multiscale-descent-error}{Lazy multiscale error bound}}\label{sec-lazy-multiscale-descent-error-proof}

We first require the following lemma.

\begin{lemma}[Inexact Lazy
Interpolation]\protect\hypertarget{lem-inexact-lazy-interpolation}{}\label[lemma]{lem-inexact-lazy-interpolation}
Let \(e=x-x^*\) be the error between \(x\) and the solution values \(x^*[i]=f^*(t[i])\), and let \(e_{(s)}\) be the error in the newly added points at the scale $s$; see \cref{def-free-variables}. In the lazy multiscale setting for an \(\lipschitz{f^*}\)-Lipchitz
function \(f^*\) on \([0, 1]\), we have the interpolated error bound at
scale \(s\),
\[
\begin{aligned}
\left\lVert e_{(s)} \right\rVert&\leq \tfrac{\lipschitz{f^*}}{2\sqrt{ 2^{S-s} }} + \lVert e_{s+1} \rVert.
\end{aligned}
\]

\begin{proof}
The proof is similar to \cref{lem-inexact-interpolation} for
greedy multiscale and differs by a factor of \(\sqrt{2}\) on the
error term \(\lVert e_{s+1} \rVert\). This is shown by modifying
\cref{sec-inexact-interpolation-proof} to exclude odd
indexes since we only look at the error for the newly interpolated points.
\end{proof}

\end{lemma}

Using \cref{lem-inexact-lazy-interpolation}
with the convergence equation (\cref{def-iterate-convergence})
gives us \[
\left\lVert e_{(s)}^{K_{s}} \right\rVert \leq q^{K_{s}}\left( \tfrac{\lipschitz{f^*}}{2\sqrt{ 2^{S-s} }} + \left\lVert e_{s+1}^{K_{s+1}} \right\rVert \right).
\]
Using \(\lVert e_{s}^{K_{s}} \rVert^2 = \lVert e_{(s)}^{K_{s}} \rVert^2 + \lVert e_{s+1}^{K_{s+1}} \rVert^2\), we have the recursion relation
\[
\begin{aligned}
\left\lVert e_{s}^{K_{s}} \right\rVert^2&\leq q^{2K_{s}}\left( \tfrac{C}{\sqrt{ 2^{S-s} }} + \left\lVert e_{s+1}^{K_{s+1}} \right\rVert \right)^2 + \left\lVert e_{s+1}^{K_{s+1}} \right\rVert^2
\end{aligned}
\]
where
\(C=\tfrac{\lipschitz{f^*}}{2}\).
We can remove the squares because, for \(a,b,c\geq 0\), having \(
a^2 \leq b^2 + c^2 \leq b^2 + 2bc + c^2= (b+c)^2\) implies \( a \leq b + c.
\)
So we have a looser bound, \[
\begin{aligned}
\left\lVert e_{s}^{K_{s}} \right\rVert &\leq q^{K_{s}}\left( \tfrac{C}{\sqrt{ 2^{S-s} }} + \left\lVert e_{s+1}^{K_{s+1}} \right\rVert \right) + \left\lVert e_{s+1}^{K_{s+1}} \right\rVert = \left( 1+q^{K_{s}} \right)\left\lVert e_{s+1}^{K_{s+1}} \right\rVert + \tfrac{Cq^{K_{s}}}{\sqrt{ 2^{S-s} }}.
\end{aligned}
\]
Applying this bound recursively for every \(s=1,\dots,S-1\) gives us
the general formula
\[
\begin{aligned}
\left\lVert e_{1}^{K_{1}} \right\rVert &\leq \left\lVert e_{S}^{K_{S}} \right\rVert\prod_{s=1}^{S-1}(1+q^{K_{s}}) + C\sum_{s=1}^{S-1} \frac{1}{\sqrt{ 2^{S-s} }}q^{K_{s}}\prod_{j=1}^{s-1}\left( 1+q^{K_{j}} \right).
\end{aligned}
\]
Performing gradient descent on the coarsest scale lets us substitute \(\lVert e_{S}^{K_{S}} \rVert \leq \lVert e_{S}^{0}  \rVert q^{K_{S}}\)
giving us the first upper bound for any plan of iterations \(K_s\).
If we now assume each scale uses the name number of iterations \(K_s=K\), \(\prod_{j=1}^{s-1}\left( 1+q^{K} \right) = (1+q^{K})^{s-1}\).
Computing the resulting geometric sum gives the final upper bound in the theorem.

\subsection{\prooftitle{lem-piecewise-function-distance}{Piecewise linear function distance}}\label{sec-piecewise-function-distance-proof}

In the following, we use the substitution \(s = m_{i}(t-t_{1})+b_{i}\) so that \(ds = m_{i} dt\), where $$
m_i = \left( f(t_{i+1}) - f(t_{i}) - g(t_{i+1}) + g(t_{i}) \right)/\Delta t
\quad\text{and}\quad
b_i = f(t_{i}) - g(t_{i}).
$$
We evaluate
\begin{align*}
\lVert \hat{f}_x - \hat{g}_y  \rVert_{2}^2 &=\textstyle\int_{t_{1}}^{t_{I}} (\hat{f}_x(t) - \hat{g}_y(t))^2 \, dt  \\
&= \textstyle\sum_{i=1}^{I-1} \textstyle\int_{t_{i}}^{t_{i+1}} \left( \left( \tfrac{{f(t_{i+1}) - f(t_{i})}}{\Delta t} - \tfrac{{g(t_{i+1}) - g(t_{i})}}{\Delta t} \right) (t-t_{i})+ (f(t_{i}) - g(t_{i}) ) \right)^2 dt \\
&=\textstyle\sum_{i=1}^{I-1}  \textstyle\int_{t_{i}}^{t_{i+1}} (m_{i}(t-t_{i})+b_{i})^2 \, dt =\textstyle\sum_{i=1}^{I-1}  \tfrac{1}{m_{i}}\textstyle\int_{b_{i}}^{m_{i}\Delta t + b_{i}} s^2 \, ds \\
&=\textstyle\sum_{i=1}^{I-1}  \tfrac{1}{m_{i}} \tfrac{1}{3} ((m_{i}\Delta t + b_{i})^3 -b_{i}^3) \\
&= \textstyle\sum_{i=1}^{I-1}  \tfrac{\Delta t}{3} ((m_{i}\Delta t + b_{i})^2 + (m_{i}\Delta t + b_{i})b_i + b_{i}^2) \\
&\leq\tfrac{\Delta t}{3} \textstyle\sum_{i=1}^{I-1}  \tfrac32\left( (f(t_{i+1}) - g(t_{i+1}))^2 + (f(t_{i}) - g(t_{i}))^2 \right) \\
&=\tfrac{\Delta t}{2} \left( 2\textstyle\sum_{i=1}^{I}  (f(t_{i}) - g(t_{i}))^2 - (f(t_{1}) - g(t_{1}))^2 - (f(t_{I}) - g(t_{I}))^2 \right) \\
&=\tfrac{\Delta t}{2} \left( 2 \left\lVert x-y \right\rVert_{2}^2 - (x[1] - y[1])^2 - (x[I] - y[I])^2 \right) \leq\Delta t\left\lVert x-y \right\rVert_{2}^2,
\end{align*}
where the first inequality uses
\(a^2+ab+b^2 \leq \tfrac32(a^2+b^2)\) with \(a = m_i\Delta t + b_i\) and
\(b=b_i\).

\subsection{\prooftitle{lem-piecewise-function-approximation}{Piecewise linear function approximation}}\label{sec-piecewise-function-approximation-proof}

By \cref{lem-lipschitz-interpolation} with substitutions \(a\mapsto t_{i+1}\),
\(b\mapsto t_{i}\), and \(t \mapsto (t-t_{i})/\Delta t\) where
\(\Delta t = t_{i+1}-t_i\) we have
\(
\tfrac{t-t_{i}}{\Delta t} t_{i+1} + \left(1-\tfrac{t-t_{i}}{\Delta t}\right) t_i
= t
\)
and
\begin{align*}
\left\lvert \tfrac{t-t_{i}}{\Delta t}f(t_{i+1}) + \left(1-\tfrac{t-t_{i}}{\Delta t}\right) f(t_i) \right. -& \left. f\left(\tfrac{t-t_{i}}{\Delta t} t_{i+1} + \left(1-\tfrac{t-t_{i}}{\Delta t}\right) t_i \right)\right\rvert \\
&\leq 2\lipschitz{f}\tfrac{t-t_{i}}{\Delta t}\left(1-\tfrac{t-t_{i}}{\Delta t}\right)\left\lvert t_{i+1} - t_i \right\rvert\\
\lvert \tfrac{f(t_{i+1}) - f(t_i)}{\Delta t}(t-t_{i}) + f(t_i) - f\left(t\right)\rvert
&\leq \tfrac{2\lipschitz{f}}{\Delta t}({t-t_{i}})\left(t_{i+1} - t\right)\\
\lvert \hat{f}(t) - f(t)\rvert
&\leq \tfrac{2\lipschitz{f}}{\Delta t}({t-t_{i}})\left(t_{i+1} - t\right).
\end{align*}

We can then bound the error, using an integration substitution of
\(s=t-t_{i}\),
\begin{align*}
\lVert \hat{f} - f \rVert_{2}^2 &= \textstyle\int_{t_{1}}^{t_{I}} (\hat{f}(t) - f(t))^2  \, dt = \textstyle\sum_{i=1}^{I-1} \textstyle\int_{t_{i}}^{t_{i+1}} (\hat{f}(t) - f(t))^2  \, dt \\
&\leq \textstyle\sum_{i=1}^{I-1} \textstyle\int_{t_{i}}^{t_{i+1}} \left(\tfrac{2\lipschitz{f}}{\Delta t}({t-t_{i}})\left(t_{i+1} - t\right)\right)^2 \, dt \\
&= \tfrac{4\lipschitz{f}^2}{\Delta t^2} \textstyle\sum_{i=1}^{I-1} \textstyle\int_{t_{i}}^{t_{i+1}} ({t-t_{i}})^2\left(t_{i+1} - t\right)^2 \, dt = \tfrac{4\lipschitz{f}^2}{\Delta t^2} \textstyle\sum_{i=1}^{I-1} \textstyle\int_{0}^{\Delta t} s^2\left(\Delta t - s\right)^2 \, ds \\
&= \tfrac{4\lipschitz{f}^2}{\Delta t^2} (I-1) \left( \tfrac{1}{3} \Delta t^2\Delta t^3 - \tfrac{1}{4} 2\Delta t \Delta t^4 + \tfrac{1}{5} \Delta t^5 \right)
= \tfrac{2}{15}\lipschitz{f}^2 \Delta t^2 (t_I -t_1).
\end{align*}

\subsection{\prooftitle{thm-problem-connection}{Continuous problem connection}}\label{sec-problem-connection-proof}

Let \(\epsilon>0\),
\(I >  \sqrt{8/15}  \lipschitz{f} \epsilon^{-1} + 1\), and
\(\left\lVert x -x^* \right\rVert_2^2 < \lipschitz{f} \epsilon/\sqrt{30}\).
This means
\(
(I-1)^{-1} < (\epsilon \sqrt{15})/(\sqrt{8}\lipschitz{f} ).
\)
We use triangle inequality with the piecewise
linear approximation \(\hat{f}_{x^*} = \hat{f}\) of \(f\) where
\(x^*[i]=f(t[i])\) is the true discretization of \(f\). This gives us
\[
\begin{aligned}
\lVert \hat{f}_x - f \rVert_2 &\leq \lVert \hat{f}_x - \hat{f}_{x^*} \rVert_2 + \lVert \hat{f}_{x^*} - f \rVert_2 = \lVert \hat{f}_x - \hat{f}_{x^*} \rVert_2 + \lVert \hat{f} - f \rVert_2.
\end{aligned}
\]
By \cref{lem-piecewise-function-distance} and our bounds on \((I-1)^{-1}\) and \(\left\lVert x -x^* \right\rVert_2^2\), we bound the first term:
\begin{align*}
\lVert \hat{f}_x - \hat{f}_{x^*} \rVert_2^2 &\leq \Delta t\left\lVert x-x^* \right\rVert_{2}^2 < \tfrac{1}{I-1}\cdot\tfrac{ \lipschitz{f} \epsilon}{\sqrt{30}}  < \tfrac{\epsilon \sqrt{15}}{\sqrt{8}\lipschitz{f}}\cdot\tfrac{ \lipschitz{f} \epsilon}{\sqrt{30}} =\tfrac{\epsilon^2}{4}.
\end{align*}
By \cref{lem-piecewise-function-approximation} and our bound on \((I-1)^{-1}\), we bound the second term:
\begin{align*}
\lVert \hat{f} - f \rVert_2^2 &\leq \tfrac{2}{15}  \lipschitz{f}^2 \Delta t^2 = \tfrac{2}{15}  \lipschitz{f}^2 \left(\tfrac{1}{I-1}\right)^2 < \tfrac{\epsilon^2}{4}.
\end{align*}
Taking square roots and adding these two bounds gives us our desired
result.

\subsection{\prooftitle{thm-expected-pgd-convergence}{Expected single-scale convergence}}\label{sec-expected-pgd-convergence-proof}

\textbf{Case 1:} \(\lVert x^*_{1} \rVert_2 = 1\).

Without loss of generality, assume (by symmetry) the solution \(x_1^*\) is such that
\(x^*_{1}[{I}] = 1\) and \(x^*_{1}[{i}]=0\) (fix a point on the sphere
aligned with the \(I\)th axis) so that \(x^*_{1}=e_{I}\), the unit vector
\(e_{I}=(0, \dots, 0, 1)\). Let
entries of our initialization be standard Gaussian
\(x^0_1[i]=g[i]\sim\mathcal{N}(0,1)\). We have
\[
\begin{aligned}
&\mathbb{E}_{g[i]\sim \mathcal{N}}\left\lVert x_1^0 - x_1^* \right\rVert_{2}^2 =\mathbb{E}_{g[i]\sim \mathcal{N}}\left\lVert g - e_{I} \right\rVert_{2}^2 =\mathbb{E}_{g[i]\sim \mathcal{N}}\left[ \textstyle\sum_{i=1}^{I-1} (g[i]-0)^2 +(g[I]-1)^2\right] \\
&=\textstyle\sum_{i=1}^{I-1} \mathbb{E}_{g[i]\sim \mathcal{N}}\left[g[i]^2\right] +\mathbb{E}_{g[I]\sim \mathcal{N}}\left[(g[I]-1)^2\right] =\left(I-1\right) + \left((1) - 2(0) +1\right) =I+1.
\end{aligned}
\]
With Gaussian concentration, we can take the square root of both sides \cite{vershynin_HighDimensionalProbability_2018}.

\textbf{Case 2:} \(\lVert x^*_{1} \rVert_2 = 0\).

Assume the solution is centred so that \(x^*_{1} = 0\in\mathbb{R}^{I}\).
Here, we have the well-known result \(
\mathbb{E}\left[ \left\lVert x^0_1 - x^*_{1}\right\rVert \right] = \mathbb{E}\left[ \left\lVert g - 0\right\rVert \right] = \sqrt{ I }
\); see \cite{vershynin_HighDimensionalProbability_2018}.

In either case, our initial error for a scaled or centred problem
goes like \(\sim \sqrt{ I }\). Since the number of points we have is
\(I\) is one plus a power of two \(I=2^S+1\), we \emph{expect} (that is, with high probability) the convergence \(
\left\lVert x_{1}^{K}-x_{1}^* \right\rVert\leq q^{K} \sqrt{ 2^S + 2 }.
\)

To ensure \(\left\lVert x_{1}^{K}-x_{1}^* \right\rVert \leq\epsilon\)
in expectation and prove \cref{cor-expected-pgd-convergence}, we need \[
\begin{aligned}
q^{K} \sqrt{ 2^S + 2 } &\leq \epsilon \iff
 \left(\tfrac{1}{2}\log\left(  2^S + 2  \right) + \log(1/\epsilon)\right)\big\slash(-\log q) \leq K.
\end{aligned}
\]

\subsection{\prooftitle{lem-cost-of-gd-vs-ms}{Cost of greedy multiscale}}\label{sec-cost-of-gd-vs-ms-proof}

The total cost of a single-scaled approach is \(
C_{\mathrm{GD}}=CK,
\)
and the cost for greedy multiscale is \(
C_{\mathrm{GM}} = \sum_{s=1}^S C_{s}K_{s}.
\)
In the case of greedy multiscale, one iteration at the finest scale \(s=1\) costs
the same for a single-scaled and multiscale approach \(C=C_1\).
Assume
\(
C_{s} = D I^p = D \left(2^{S-s+1}+1\right)^p
\)
for some constant \(D>0\), power \(p\geq 1\), and coarsest scale \(S\),
where there are \(2^{S}+1\) points at the finest scale \(s=1\).
We wish to lower bound the ratio \[
\tfrac{C_{s}}{C_{s+1}} = \tfrac{ D \left(2^{S-s+1}+1\right)^p}{D \left(2^{S-(s+1)+1}+1\right)^p}= \left(\tfrac{2^{S-s+1}+1}{2^{S-s}+1}\right)^p
\]
for \(1\leq s \leq S-1\). Letting \(x=S-s\), the function
\(g(x)=\frac{ 2^{x+1}+1 }{ 2^{x}+1}\) is increasing so it is minimized
at \(x=1\) on \(1\leq x = S-s \leq S - 1\). This means
\(g(x) \geq g(1) = 5/3\) and, since \(p\geq 1\),
\[
\tfrac{C_{s}}{C_{s+1}} \geq \left(\tfrac{5}{3}\right)^p \geq \tfrac{5}{3} \implies C_{1}\geq \tfrac{{5}}{3}C_{2}\geq \left( \tfrac{5}{3} \right)^2 C_{3}\geq \dots \geq \left( \tfrac{5}{3} \right)^{S-1} C_{S}.
\]
Multiplying through by \((3/5)^{S-1}\) and noting \(C=C_1\) gives us our final inequality and completes the proof with \(
C_{\mathrm{GM}} = \sum_{s=1}^S C_{s}K_{s} \leq C_{1}\sum_{s=1}^S \left( \frac{3}{5} \right)^{s-1} K_{s}.
\)

To show this is less than \(C_{\mathrm{GD}}\) with our plan of one iteration \(K_s=1\) at all scales except the finest \(K_1 = K-2\), observe, \[C_{\mathrm{GM}} = C(K-2) + C \textstyle\sum_{s=2}^S \left( \frac{3}{5} \right)^{s-1} = C(K-2) + C\frac{3/5 - (3/5)^S}{1-(3/5)}< CK = C_{\mathrm{GD}}.\]

We stress that the lower bound \(\tfrac{C_{s}}{C_{s+1}} = g(S-s)^p \geq \tfrac{5}{3}\) is quite pessimistic. The function \(g\) rapidly converges to \(2\) as \(S-s\) gets large (the fine scales have many points): \(g(10)>1.999\). Additionally, any power \(p>1\) will further inflate this bound. This implies the multiscale method will often be much cheaper in practice than this estimate yields.

\subsection{\prooftitle{cor-greedy-cost}{Sufficient conditions for greedy multiscale to be better}}\label{sec-greedy-cost-proof}

From \cref{lem-cost-of-gd-vs-ms}, we know greedy multiscale is cheaper than a single-scaled approach in this setting.
We need to show the upper bound on the
expected error for the single-scaled approach
(right-hand side of \cref{thm-expected-pgd-convergence}) is larger than the corresponding bound for greedy multiscale (right-hand side of \cref{thm-expected-greedy-convergence}). Assuming our plan for the number of iterations,
\(\sum_{t=2}^s K_s = s-1\) so our bound for greedy multiscale becomes
\(
q^{K_{1}}\sqrt{ 2 ^{S+1}}\left(q^{S-1}+\tfrac{\lipschitz{f^*}}{{2^{S+2}q}} \tfrac{ 2 q - (2 q)^{S} }{1 - 2 q}  \right).
\)

Assume \(0<q<\frac{1}{2}\), and
the number of scales \(S\) is
\(\geq 4\) and
\(
2^S \geq  \tfrac{\lipschitz{f^*} \sqrt{2}}{2(1-2q)(q^2-\sqrt{2}q^3)}.
\)
Because
\(S-1\geq 3 > 0\) and \(0 < q < 1/2\), we have \(1 \overset{}{>} 1 - (2q)^{S-1} > 0.\)
We also have
\[
\begin{aligned}
\frac{1}{q^2\sqrt{2^{-1} + 2^{-S}} - q^{S-1}} &\overset{(\text{a})}{<} \frac{1}{q^2\sqrt{2^{-1}} - q^{S-1}} = \frac{q^{-2}\sqrt{2}}{1 - \sqrt{2}q^{S-3}} \\
&\overset{(\text{b})}{\leq} \frac{q^{-2}\sqrt{2}}{1 - \sqrt{2}q} = \frac{\sqrt{2}}{q^2 - \sqrt{2}q^3}.
\end{aligned}
\]
We have the inequality \((\text{a})\) since \(2^{-S}>0\), and \((\text{b})\) since
\(S-3\geq 1\) implies \(q^{S-3}\leq q\). This gives us
\begin{align*}
2^S &\geq  \frac{\lipschitz{f^*} \sqrt{2}}{2(1-2q)(q^2-\sqrt{2}q^3)} \\
\text{(by \(1>1-(2q)^{S-1}\) and \((\text{b})\))} \quad 2^S & > \frac{\lipschitz{f^*} (1 - (2q)^{S-1})}{2(1-2q)(q^2\sqrt{2^{-1} + 2^{-S}} - q^{S-1})} \\
(q^2\sqrt{2^{-1} + 2^{-S}} - q^{S-1}) 2^S & > \frac{\lipschitz{f^*} (1 - (2q)^{S-1})}{2(1-2q)} \\
q^2\sqrt{2^{-1} + 2^{-S}} & > q^{S-1} + \frac{\lipschitz{f^*} (1 - (2q)^{S-1})}{2^{S+1} (1-2q)}\frac{2q}{2q} \\
q^2\sqrt{\frac{2^{S} + 2}{2^{S+1}}} & > q^{S-1} + \frac{\lipschitz{f^*} (2q - (2q)^{S})}{2^{S+2} q(1-2q)} \\
q^K\sqrt{2^{S} + 2} & > q^{K-2}\sqrt{2^{S+1}}\left(q^{S-1} + \frac{\lipschitz{f^*} (2q - (2q)^{S})}{2^{S+2} q(1-2q)}\right).
\end{align*}

Since \(K_1 = K-2\), this completes the proof.

\subsection{\prooftitle{lem-pgd-vs-lazy-multiscale}{Cost of lazy multiscale}}\label{sec-cost-of-gd-vs-lazy-ms-proof}

We assume a similar setup to \cref{lem-cost-of-gd-vs-ms}: \(C_{\mathrm{GD}}=CK=D \left(2^{S}+1\right)^p\) as before.
Now we compare against lazy multiscale with cost \(C_{\mathrm{LM}}=\sum_{s=1}^S C_{s}K_{s}\) and \(C_s=DI^p\), where \(I=2^{S-s}\)
for \(1 \leq s \leq S - 1\) (decent on free variables only), and \(C_S = C3^p\) (start with \(3\) points).
So we have the ratios,
\[
\tfrac{C}{C_{s}} =  \tfrac{ D \left(2^{S}+1\right)^p}{D \left(2^{S-s}\right)^p} > \tfrac{ \left(2^{S}\right)^p}{ \left(2^{S-s}\right)^p}= \left(2^{s}\right)^p > 2^{s}
\ \ \text{and}\ \
\tfrac{C}{C_{S}} =  \tfrac{ D \left(2^{S}+1\right)^p}{D 3^p} > \tfrac{ \left(2^{S}\right)^p}{ 3^p} > \tfrac{2^{S}}{3}.
\]
Performing the same number of iterations \(K_s = K'\) at each scale,
we have,
\[
\begin{aligned}
C_{\mathrm{LM}} &= \textstyle \sum_{s=1}^S C_{s}K_{s} \leq C \left(\textstyle \sum_{s=1}^{S-1} 2^{-s}K_{s} + 2^{-S}3K_{S}\right) \\
&\leq C K' \left(\textstyle \sum_{s=1}^{S-1} 2^{-s} + 2^{-S}3\right) = CK'\left(1-2^{-S}2 + 2^{-S}3\right) = CK'\left(1+ 2^{-S}\right).
\end{aligned}
\]

To ensure this is less than \(CK\), we need
\(
K' < \frac{1}{1+ 2^{-S}} K.
\)
The right side factor is maximized at \(S=2\) on \(2\leq S\), so we can
set
\(
K' = \lceil \frac{1}{1+ 2^{-2}} K \rceil - 1 = \left\lceil \frac{4}{5} K \right\rceil - 1
\)
and ensure that \(C_{\mathrm{LM}} < C_{\mathrm{GD}}\).

\subsection{\prooftitle{cor-lazy-cost}{Sufficient conditions for lazy multiscale to be better}}\label{sec-lazy-cost-proof}

From \cref{lem-pgd-vs-lazy-multiscale}, we know
lazy multiscale will be cheaper than a single-scaled approach. It remains to show that the bound on the error at the final scale for lazy multiscale on the right-hand side of
\cref{eq-lazy-multiscale-error-bound} is less than bound for the single-scaled approach
\cref{eq-pgd-expected-error}.

Let \(C=\lipschitz{f^*}/2\), \(h=1+q^{K'}\),
and \(g=q^{-K/5-1}\). To ensure first line bellow is valid, we cannot
have \(\sqrt{2}/h=1\). Otherwise, the logarithm in the denominator
becomes zero. Moreover, we assume \(K>5(\log_q(\sqrt{2}-1)-1)/4\) so
that \(\sqrt{2}/h>1\) and we can write inequality \((1)\) below ensuring \(\log(\sqrt{2}/h) > 0\).
\begin{align*}
S&\geq {\log\left( \tfrac{g}{h}\left( 2+C\tfrac{1}{h\sqrt{ 2 } -1} \right) \right)} \Big/{\log\left( \tfrac{\sqrt{ 2 }}{h} \right)} \\
\log\left( \tfrac{\sqrt{ 2 }}{h} \right)S& \overset{(1)}{\geq} \log\left( \tfrac{g}{h}\left( 2+C\tfrac{1}{h\sqrt{ 2 } -1} \right) \right) \\
\sqrt{ 2^S } &\geq gh^{S-1}\left( 2+C\tfrac{1}{h\sqrt{ 2 } -1} \right) \\
\sqrt{ 2^S+2 } &\overset{(2)}{>} gh^{S-1}\left( 2+C\tfrac{1-(\sqrt{ 2 } h)^{-(S-1)}}{h\sqrt{ 2 } -1}\right) \\
q^K\sqrt{ 2^S+2 } & > q^{4K/5-1}\left( 2h^{S-1}+C\tfrac{h^{S}\sqrt{ 2^S }-h\sqrt{ 2 }}{\sqrt{ 2^S }h(h\sqrt{ 2 } -1)}\right)\\
q^K\sqrt{ 2^S+2 } & \overset{(3)}{>} q^{\lceil 4K/5\rceil-1}\left( 2h^{S-1}+C\tfrac{h^{S}\sqrt{ 2^S }-h\sqrt{ 2 }}{\sqrt{ 2^S }h(h\sqrt{ 2 } -1)}\right)\\
q^K\sqrt{ 2^S+2 } & > q^{K'}2(1+q^{K'})^{S-1}+q^{K'}\tfrac{\lipschitz{f^*}}{2} \tfrac{(1+q^{K'})^{S}\sqrt{ 2^S }-(1+q^{K'})\sqrt{ 2 }}{\sqrt{ 2^S }(1+q^{K'})(\sqrt{ 2 }(1+q^{K'}) -1)}
\end{align*}

For inequality \((2)\), we use the fact that
\(\sqrt{ 2^S+2 }>\sqrt{ 2^S }\) and \(0<1-(\sqrt{ 2 }h)^{-(S-1)}<1\) to
relax the previous line. For inequality \((3)\), we use the fact that
\(0<q<1\) so that possibly multiplying by at most a factor of \(q\) only
makes the right-hand side smaller.

\section{Motivating example details}\label{sec-motivating-example-details}

This section explains additional details for the motivating example from \cref{sec-motivating-example,sec-motivating-example-numerics}.

For \cref{fig-multi-vs-single-time}, we compare total runtime for \(S=3, 4,\dots,12\), which yields finest discretizations with \(I_1=2^S+1\) points.
We initialize by uniformly discretizing the polynomial \(p(t) = -2.625t^4 - 1.35t^3 + 2.4t^2 + 1.35t + 0.225\) on \([-1,1]\) to obtain the initial density approximation \(x_S[i]=f(t_S[i])\Delta t_S\) at the coarsest scale.
At each scale \(s > 1\), we perform one (\(K_s=1\)) iteration; at the finest scale (\(s=1\)), and we iterate until the objective is within \(5\%\) of its optimal value. We generate Legendre measurements for \(m=1,\dots,5\) using the measurement operator $A_1$ from \cref{sec-motivating-example} and add \(5\%\) Gaussian noise to obtain \(y\). The regularization parameter \(\lambda = 10^{-4}\) balances data fidelity against smoothness.

For \cref{fig-iterate-progression}, we run the motivating example with \(2^{10}+1\) points at the finest scale \(s=1\). The top-left plot shows every 25th iterate \(k\) starting with the initialization at \(k=0\) and ending at \(k=250\) when running projected gradient descent at the finest scale (\(s=1\)). The top-right plot shows every 250th iterate \(k\), ending with the final iterate \(k=4988\) which satisfies the stopping criteria \(\tilde{\mathcal{L}}_1(x_1^{k}) < 0.00024\). The bottom-left plot shows \(x_s^k\) at each scale from the coarsest scale \(s=10\) to the second finest \(s=2\), and the iterate at the finest scale \(s=1\) after \(k=25\) iterations. We can make a fair comparison between \(x_1^{25}\) on the top-left and bottom-left plots since they are both the result of running \(k=25\) iterations of projected gradient descent at the finest scale. Using multiscale results in a smoother iterate since \(x_1^0\) was constructed from the chain of interpolations, rather than the completely random initialization used by the single scale approach. Because of this, the bottom-left plot highlights multiscale only needs \(k=387\) iterations at the finest scale to also achieve the same stopping criteria \(\tilde{\mathcal{L}}_1(x_1^{k}) < 0.00024\).

For \cref{fig-loss-convergence-in-time}, comparing iterations at different scales is not fair since we expect one projected gradient step to be cheaper at coarser scales. The loss \(\tilde{\mathcal{L}}_s(x_s^k)\) for the multiscale method is calculated at the respective scale \(s\) of the iterate, whereas the single scale only calculates the the loss \(\tilde{\mathcal{L}}_1(x_1^k)\) at the finest scale \(s=1\). This explains why the loss fluctuates up and down during interpolations for coarser scales \(s\) when the approximation \(\tilde{\mathcal{L}}_s\) for \(\mathcal{L}\) is worse. The coarsest scales also take longer than expected because of the additional function call overhead. Despite these drawbacks, multiscale more than makes up for these by reducing the loss significantly faster once iterating on the finer scales. This is make clear by the right plot in \cref{fig-loss-convergence-in-time}.

It could be argued that for optimal convergence rates with a first order method, we should not be using a fixed stepsize based on the global Lipschitz constant \cite{nesterov-smooth-2018}. Using accelerated methods would almost certainly accelerate convergence and not require almost \(5000\) iterations. But any sort of acceleration could also be used to speed up convergence for the multiscale method which already has a leg-up since the fine scale iterations start closer to the solution.

\section{Additional density unmixing with Tucker-1 tensor factorization details}\label{additional-tensor-details-and-reference}

\subsection{Tucker-1 tensor
decomposition}\label{sec-common-decompositions}

A tensor decomposition is a factorization of a tensor into multiple
(usually smaller) tensors, that can be recombined into the original
tensor. \Cref{sec-numerical-experiments} uses the Tucker-1
decomposition defined in \cref{def-tucker-1-decomposition}.

\begin{definition}[Tucker-1
Decomposition]\protect\hypertarget{def-tucker-1-decomposition}{}\label[definition]{def-tucker-1-decomposition}
A rank-\(R\) Tucker-\(1\) decomposition of a tensor
\(Y\in\mathbb{R}^{I_1\times\dots\times I_N}\) produces a matrix
\(A\in\mathbb{R}^{I_1\times R}\), and core tensor
\(B\in\mathbb{R}^{R\times I_2\times\dots\times I_N}\) such that
\begin{equation}\phantomsection\label{eq-tucker-1}{
Y[i_1,\dots,i_N] = \sum_{r=1}^{R} A[i_1,r] B[r, i_2,\dots,i_N]
}\end{equation}
entry-wise or more compactly,
\[
Y = B\times_1 A.
\]

\end{definition}

Tensor decompositions are not necessarily unique. It should be clear
that scaling one factor by \(c\neq 0\) and dividing another by \(c\)
yields the same original tensor. Furthermore, slices can be permuted
without affecting the the original tensor. Up to these manipulations,
for a fixed rank, there exist criteria that ensures their decompositions
are unique
\cite{kolda_TensorDecompositionsApplications_2009,kruskal_three-way_1977,bhaskara_uniqueness_2014}.

\subsection{Formulating density unmixing as a Tucker-1 decomposition problem}

In both the synthetic and geological data, the task is to recover
mixtures of probability densities. We accomplish this task by
formulating a discretized version of the problem as a tensor
decomposition problem.

In the case of the synthetic example in
\cref{sec-synthetic-probability-unmixing}, we have access to the
true probability density mixtures. For the real-world example in
\cref{sec-geological-probability-unmixing}, the continuous
probability density mixtures are estimated from sampling using kernel
density estimation \cite{chen_tutorial_2017,graham_tracing_2025}.

We use the
\(2\)-norm/Frobenius loss in the formulation of the continuous decomposition problem \cref{eq-continuous-decomposition-problem}, but alternative losses
can be used such as the KL divergence and their corresponding discrete
versions based on how the error between the data and model are
distributed \cite{gillis_nmf_2020}.
The constraints \cref{eq-continuous-decomposition-constraints-a,eq-continuous-decomposition-constraints-b} ensure each mixture in the model is indeed a probability
density function.

We interpret the tensors
\(Y\in\mathbb{R}_+^{I\times K_1\times\dots \times K_N}\) and
\(B\in\mathbb{R}_+^{R\times K_1\times\dots \times K_N}\) as samples of
the underlying continuous probability density functions
\[
Y[i, k_1,\dots,k_N] = y_i(x[k_1,\dots,k_N])\Delta x,\ \text{and}\ B[r, k_1,\dots,k_N] = b_r(x[k_1,\dots,k_N])\Delta x,
\]
with a uniformly spaced grid\footnote{If the domain \(\mathcal{D}\) is
  unbounded, we may restrict the domain to the support of \(y_i\), or
  where \(y_i(x)\geq \epsilon\) for some \(\epsilon > 0\).}
\(\{x[k_1,\dots,k_N]\}\subset\mathcal{D}\). We scale by the volume
element of the grid,
\[
\Delta x = \prod_{n=1}^N \Delta_n x=\prod_{n=1}^N (x[1,\dots,1,\underbrace{2}_n,1,\dots 1 ]- x[1,\dots,1,\underbrace{1}_n,1,\dots 1]),
\]
to ensure the tensors are normalized, \(Y\in\Delta_{K_1\dots K_N}^I\)
and \(B\in\Delta_{K_1\dots K_N}^R\). This leads to the convenient
notation of switching from integrals to summations when we discretize;
\[
1 = \int_{\mathcal{D}} y_i(x) dx \quad \text{discretizes to} \quad 1 \approx \sum_{k_1,\dots,k_N=1}^{K_1\dots K_N} Y[i, k_1,\dots,k_N] \Delta x.
\]

In higher dimensions when \(N\) is large, it can be expensive to compute
the full \(N\)-dimensional kernel density estimation\footnote{Despite
  the existence of fast algorithms for kernel density estimation
  \cite{obrien_fast_2016}, it can still be cheaper to compute \(N\)
  one-dimensional kernel density estimations than one \(N\)-dimensional
  kernel density estimation.} for each \(b_r\) and memory intensive to
store the full \(N\) dimensional tensor \(Y\). We can instead
approximate the full distribution
\(y_i:\mathcal{D}\subseteq \mathbb{R}^N\to\mathbb{R}_+\) as a product
distribution of \(J=N\) one-dimensional distributions
\(y_i^j:\mathcal{D}^j\subseteq\mathbb{R} \to \mathbb{R}_+\):
\[
y_i(x) = \prod_{j=1}^J y_i^j (x[j]),\quad\text{where} \quad \mathcal{D} = \mathcal{D}^1\times\dots\times \mathcal{D}^J,
\]
and similarly with \(b_r\) and \(b_r^j\). This lets us discretize the
probability density functions as \(3\)rd order tensors regardless of the
number of dimensions \(N\),
\[
Y[i, j, k] = y_i^j(x_{k}^j)\Delta x^j,\quad\text{and}\quad B[r,j,k] = b_r^j(x_{k}^j)\Delta x^j,
\]
with a 1D grid \(\{x_k^j\}\subset\mathcal{D}^j\) for each \(j\in[J]\).
The constraint on \(B\) also gets similarly modified to
\[
B\in\Delta_{K}^{RJ}=\Set{B\in\mathbb{R}_+^{R\times J \times K} | \forall (r,j)\in [R]\times[J],\, \sum_{k=1}^{K} B[r,j,k] = 1 }.
\]

We use the full \(3\)-dimensional representation of the probability
density mixtures for the synthetic data in
\cref{sec-synthetic-probability-unmixing}, and use the compressed
representation of the \(7\)-dimensional mixtures for the real word data
in \cref{sec-geological-probability-unmixing}.

\subsection{Solution algorithm}\label{sec-base-algorithm}

Let
\[
f(A, B) := \frac{1}{2}\left\lVert B\times_1 A - Y\right\rVert _F^2
\]
be the objective function we wish to minimize in
\cref{eq-constrained-least-squares}. Following Xu and Yin
\cite{xu_BlockCoordinateDescent_2013}, the general approach we take
to minimize \(f\) is to apply block coordinate descent using each factor
as a different block. Let \(A^t\) be the \(t\)th iteration of \(A\) and
let
\begin{subequations}
\begin{align}
f_A^t(A) &:= \frac{1}{2}\left\lVert B^t \times A- Y \right\rVert _F^2 \label{eq-block-objective-functions-a}\\
f_B^t(B) &:= \frac{1}{2}\left\lVert B \times A^t- Y \right\rVert _F^2\label{eq-block-objective-functions-b}
\end{align}
\end{subequations}
be the (partially updated) objective function at iteration \(t\) for the
factor \(A\), and similarly with \(B\).

Given initial factors \(A^0\) and \(B^0\), we alternate through the
factors and perform the updates
\begin{subequations}
\begin{align}
A^{t+1} &\leftarrow P_{\Delta_{R}^I}\left( A^{t} - \frac{1}{\smoothness{f_A^t}}\nabla f_A^t(A^{t})\right) \label{eq-proximal-explicit-a}\\
B^{t+1} &\leftarrow P_{\Delta_{K_1\dots K_N}^R}\left( B^{t} - \frac{1}{\smoothness{f_B^t}}\nabla f_B^t(B^{t})\right) \label{eq-proximal-explicit-b}
\end{align}
\end{subequations}
for \(t=1,2,\dots\) until some convergence criterion is satisfied. In
our case, we iterate until the relative error,
\[
\frac{\left\lVert B^t \times A^t- Y \right\rVert _F}{\left\lVert Y \right\rVert _F},
\]
or mean relative error,
\[
\frac{1}{I \cdot K_1 \cdot \ldots \cdot K_N} \sum_{k_1,\dots,k_N}\frac{ \left\lvert (B^t \times A^t)[k_1,\dots,k_N] - Y[k_1,\dots,k_N] \right\rvert}{Y[k_1,\dots,k_N]},
\]
is less than some specified tolerance \(\delta\). In practice to avoid dividing by small entries of \(Y\), we may take the mean relative error only on the subarray of \(Y\) where entries are larger than some threshold. Given \(Y\) is in the constraint set \(\Delta_{K}\)

We choose a stepsize of \(\alpha=1/\smoothness{f_A^t}\), since it is a
sufficient condition to guarantee
\(f_A^t(A^{t+1}) \leq f_A^t(A_A^{t})\), but other stepsizes can be used
in theory \cite[Sec.~1.2.3]{nesterov_NonlinearOptimization_2018}.
This choice also agrees with the best stepsize given in
\cref{lem-descent-lemma} since the least-square loss in
\cref{eq-block-objective-functions-a} is
\(\smoothness{f_A^t}\)-smooth and
\(\convexity{f_A^t}\) strongly convex where \(\convexity{f_A^t}=\smoothness{f_A^t}\). A similar idea
holds with \(B\) in \cref{eq-block-objective-functions-b}.

\section{Synthetic data setup}\label{sec-synthetic-data-setup}

This section details the data used in
\cref{sec-synthetic-probability-unmixing}. The main setup is
provided here, and the full file can be viewed at\newline\texttt{julia\_experiments/distribution\_unmixing\_synthetic.jl} in this paper's GitHub
repository \cite{Richardson_multiscale_paper_code}.

\begin{jllisting}
using Distributions

source1a, source1b, source1c = Normal(4, 1), Uniform(-7, 2), Uniform(-1, 1)
source2a, source2b, source2c = Normal(0, 3), Uniform(-2, 2), Exponential(2)
source3a, source3b, source3c = Exponential(1), Normal(0, 1), Normal(0, 3)

source1 = product_distribution([source1a, source1b, source1c])
source2 = product_distribution([source2a, source2b, source2c])
source3 = product_distribution([source3a, source3b, source3c])

sources = (source1, source2, source3)
\end{jllisting}

We generate the following \(5\times 3\) mixing matrix
\begin{jllisting}
p1 = [0.0, 0.4, 0.6]
p2 = [0.3, 0.3, 0.4]
p3 = [0.8, 0.2, 0.0]
p4 = [0.2, 0.7, 0.1]
p5 = [0.6, 0.1, 0.3]

A_true = hcat(p1,p2,p3,p4,p5)'
\end{jllisting}
and use it to construct \(5\) mixture distributions.

\begin{jllisting}
distribution1 = MixtureModel([sources...], p1)
distribution2 = MixtureModel([sources...], p2)
distribution3 = MixtureModel([sources...], p3)
distribution4 = MixtureModel([sources...], p4)
distribution5 = MixtureModel([sources...], p5)
distributions =
    [distribution1, distribution2, distribution3, distribution4, distribution5]
\end{jllisting}

These are discretized into \(65\times 65\times 65\) sample tensors, and
stacked into a \(5\times 65\times 65\times 65\) tensor \(Y\). We
normalize the \(1\)-slices so that they sum to one.

\begin{jllisting}
sinks = [pdf.((d,), xyz) for d in distributions]
Y = cat(sinks...; dims=4)
# reorder so the first dimension indexes mixtures
Y = permutedims(Y, (4,1,2,3))
Y_slices = eachslice(Y, dims=1)
correction = sum.(Y_slices) # normalize slices to 1
Y_slices ./= correction
\end{jllisting}

The options used for both the single and multi scaled factorization are the following.

\begin{jllisting}
options = (
  rank=3,
  momentum=false,
  model=Tucker1,
  tolerance=(0.05, 1e-6),
  converged=(MeanRelError, ObjectiveValue),
  do_subblock_updates=true,
  constrain_init=true,
  constraints=[nonnegative!, simplex_rows!],
  stats=[Iteration, ObjectiveValue, GradientNNCone, RelativeError, MeanRelError],
  mean_rel_error_tol = 1e-4,
  maxiter=50
)
\end{jllisting}

At each scale, the algorithm will continue to iterate until one of the following three occur. The mean relative error is below \(5\%\), the objective value is below \(10^{-6}\), or \(50\) iterations have occurred. At this point, the iteration at the next smaller scale will start.

To create \cref{fig-time-vs-coarse-iterations}, we modify the options so that we only use the objective value at the convergence criteria. This allows us to make a fairer comparison between different number of coarse iterations, and the analysis in the paper. The code for this figure can be found in \texttt{julia\_experiments/distribution\_unmixing\_synthetic\_figures.jl}.

\subsection{Geological data factorization details}\label{sec-geo-numerics-details}

 We discretize the
densities with \(K=2^{10} + 1 = 1025\) points to obtain an input tensor
\(Y\in\mathbb{R}_+^{20\times 7 \times 1025}\) and normalize the depth
fibres so that \(\sum_{k\in[K]} Y[i,j,k] = 1\) for all \(i\in[20]\) and
\(j\in[7]\).

We run the multiscale factorization algorithm with the following call
\begin{jllisting}
multiscale_factorize(Y; continuous_dims=3, options...)
\end{jllisting}
from \texttt{BlockTensorFactorization}.
The third dimension is specified as continuous since each
depth fibre \(Y[i, j, :]\) is a discretized continuous probability
density function.

This is compared to the regular factorization algorithm
\begin{jllisting}
factorize(Y; options...)
\end{jllisting}
using the Julia package \texttt{BenchmarkingTools}. This runs the
algorithm as many times as it can within a default time window. Note
that a new random initialization is generated for each run. After
running the following,
\begin{jllisting}
using BenchmarkingTools

# Run each function once so they compile
factorize(Y; options...);
multiscale_factorize(Y; continuous_dims=3,options...);

benchmark1 = @benchmark factorize(Y; options...)
display(benchmark1)

benchmark2 = @benchmark multiscale_factorize(Y; continuous_dims=3,options...)
display(benchmark2)
\end{jllisting}
we observe the two benchmarks shown in \cref{sec-geological-probability-unmixing}.

We run the multiscale and single scale factorization algorithms with the following options.

\begin{jllisting}
options = (
    rank=3,
    momentum=false,
    do_subblock_updates=false,
    model=Tucker1,
    tolerance=(0.12),
    converged=(RelativeError),
    constrain_init=true,
    constraints=[l1scale_average12slices! ∘ nonnegative!, nonnegative!],
    stats=[Iteration, ObjectiveValue, GradientNNCone, RelativeError],
    maxiter=200
)
\end{jllisting}

We use the same convergence criteria at each scale and iterate until the
relative error between the input \(Y\) and our model \(X\) is at most
\(12\%\), or until \(200\) iterations have passed. We use \(12\%\)
because this is roughly the error Graham et.\ al.~observe in the final
factorization \cite{graham_tracing_2025}, suggesting this is the
roughly the smallest amount of error we can expect in this
factorization.

\end{document}